\documentclass[12pt, a4paper]{elsarticle}
\usepackage{amsthm,amsmath,amssymb,extsizes,easymat,graphicx,hyperref,easybmat}
\usepackage{mathdots}

\usepackage[all]{xy}

\usepackage[active]{srcltx}
\usepackage{MnSymbol}
\usepackage{rotating}
\usepackage{mathrsfs,color}
\usepackage{cases}
\DeclareMathOperator{\rank}{rank}

\DeclareMathOperator{\diag}{diag}
\DeclareMathOperator{\orb}{O}

\DeclareMathOperator{\rev}{rev}
\DeclareMathOperator{\POL}{POL}
\DeclareMathOperator{\PEN}{PEN}

\DeclareMathOperator{\vect}{vec}
\DeclareMathOperator{\dist}{dist}

\usepackage{tcolorbox}
\tcbuselibrary{skins}
\tcbuselibrary{breakable}
\tcbset{shield externalize}

\newcommand{\sdotsss}%
{\text{\raisebox{-2.2pt}{$\cdot\,$}%
 \raisebox{1.7pt}{$\cdot$}%
\raisebox{5.6pt}{$\,\cdot$}}}

\renewcommand{\le}{\leqslant}
\renewcommand{\ge}{\geqslant}

\newtheorem{theorem}{Theorem}[section]

\newtheorem{lemma}[theorem]{Lemma}

\newtheorem{definition}[theorem]{Definition}

\newtheorem{example}[theorem]{Example}

\newcommand{\hide}[1]{}

\begin{document}

\title{Distance to nearest skew-symmetric matrix polynomials of bounded rank \vskip 0.5cm
\small{\it Dedicated to Paul Van Dooren on the occasion of his 75th birthday}}


\date{}

\author[um]{Andrii Dmytryshyn}
\ead{andrii@chalmers.se}

\author[cm]{Froil\'an M. Dopico}
\ead{dopico@math.uc3m.es}

\author[um]{Rakel Hellberg}
\ead{rakel.hellberg@outlook.com}

\address[um]{Department of Mathematical Sciences, Chalmers University of Technology and University of Gothenburg, Gothenburg, Sweden.}

\address[cm]{Universidad Carlos III de Madrid, ROR: https://ror.org/03ths8210, Departamento de Matem\'aticas, Avenida de la Universidad 30, 28911, Legan\'es, Spain.}

\begin{abstract}
We propose an algorithm that approximates a given matrix polynomial of degree $d$ by another skew-symmetric matrix polynomial of a specified rank and degree at most $d$. The algorithm is built on recent advances in the theory of generic eigenstructures and factorizations for skew-symmetric matrix polynomials of bounded rank and degree. Taking into account that the rank of a skew-symmetric matrix polynomial is even, the algorithm works for any prescribed even rank greater than or equal to $2$ and produces a skew-symmetric matrix polynomial of that exact rank. We also adapt the algorithm for matrix pencils to achieve a better performance. Lastly, we present numerical experiments for testing our algorithms and for comparison to the previously known ones.
\end{abstract}


\maketitle

\section{Introduction}

Matrix nearness problems—quantifying the minimal perturbation needed to endow a matrix or matrix-valued function with a desired property—are fundamental in numerical linear algebra \cite{GuLu25,HighamNear88,KrVo15}. Among these, determining the distance to the nearest singular matrix polynomial (including pencils) has garnered significant attention \cite{ByHM98,DaBo23,DoNN23,GiHL17,GnGu23,GNNP24,GuLM17} due to its relevance in control theory, system identification, and the analysis of differential-algebraic equations. 



In this work, we focus on a structured and more general variant of the distance to the singularity problem, namely, on computing the distance of a given matrix polynomial of degree $d$ to the nearest skew-symmetric matrix polynomial of a given bounded rank and degree at most $d$ and on computing such a nearest polynomial. Note that we impose a stronger condition than just the singularity on such nearest polynomial and that the distance to the singularity problem is a particular case of the problem we consider. Skew-symmetric matrix polynomials arise naturally in Hamiltonian systems, gyroscopic dynamics, and other physical models where antisymmetric structure encodes conservation laws or symmetries. 
Moreover, canonical forms of skew-symmetric matrix pencils have been studied in \cite{Rodm07, Thom91}, low-rank perturbations in \cite{Batz16}, miniversal deformations in \cite{Dmyt16}, codimensions of their orbits were computed in \cite{DmKS13}, and their stratifications were developed in \cite{DaDm25,Dmyt17,DmKa14}.

Two recent methodological advances provide powerful tools for addressing the distance to the singularity problem for matrix polynomials (among the other challenges they also address). The first, introduced by Gnazzo and Guglielmi in \cite{GnGu23,GnGu25}, considers the distance to singularity for general nonlinear matrix-valued functions. Their approach reformulates the problem as a non-convex optimization task, encoding singularity as a constraint. 
%
Another approach, called the Riemann-Oracle \cite{GNNP24}, developed by Gnazzo, Noferini, Nyman, and Poloni, offers a general-purpose Riemannian optimization strategy for a broad class of matrix nearness problems. 

In this paper, we follow the approach of \cite{DDDG26} and build on recent advances in the theory of generic eigenstructures and factorizations for skew-symmetric matrix polynomials of bounded rank and degree \cite{DmDV23,DmDV26}. Notably, our approach is completely different from those in \cite{GnGu23,GnGu25,GNNP24} and allows us to consider the distance to matrix polynomials of any specified rank in the same way as the distance to the singularity. 



During recent years, the problem of determining the generic eigenstructures for sets of matrix polynomials (including pencils) with fixed grade and bounded rank, has been intensively investigated. The description of such sets as a union of closures of a few ``generic sets'' of matrix polynomials with simple eigenstructures has been given for general matrix polynomials \cite{DmDo17} (see also \cite{DeDo08} for the result on pencils), for skew-symmetric matrix polynomials \cite{DeDD24,DmDo18} (including pencils), as well as for symmetric matrix polynomials of odd grade \cite{DeDD20b} (see also \cite{DeDD20a}  for the result on pencils). Moreover, the cases of matrix pencils that are $T$-palindromic and $T$-alternating \cite{DeTe18}, and Hermitian \cite{DeDD22} are also known. In this paper we develop a parametrization of the generic sets for skew-symmetric matrix polynomials with fixed grade and bounded rank in terms of factorizations with very specific properties, similarly to how it was done in \cite{DmDV23} for general matrix polynomials. This parametrization allows us to develop a simple and efficient algorithm for solving the distance problem mentioned above.



The rest of the paper is organized as follows. We begin by providing, in Section~\ref{sect.prempolys}, the necessary background on skew-symmetric matrix polynomials, introducing key definitions and properties that will be used throughout the paper. Section \ref{sec.genfacts} presents the announced parametrization of the generic sets for skew-symmetric matrix polynomials with fixed grade and bounded rank in terms of factorizations. In Section~\ref{sec:poldist}, we present a general solution to the distance problem for skew-symmetric matrix polynomials. Then, in Section~\ref{pencils}, we review the relevant background on skew-symmetric matrix pencils, i.e., degree-one skew-symmetric matrix polynomials. In Section~\ref{sec:distpen}, we adapt our algorithm from Section~\ref{sec:poldist} to the case of skew-symmetric matrix pencils and improve its computational efficiency. Finally, Section~\ref{sect.num} presents numerical experiments.

\section{Skew-symmetric matrix polynomials}
\label{sect.prempolys}

We consider skew-symmetric $m\times m$ matrix polynomials $P(\lambda)$ of grade $d$, i.e., of degree less than or equal to $d$, over the field of complex numbers~$\mathbb C$:
\begin{equation*}
P(\lambda) = \lambda^{d}A_{d} + \dots +  \lambda A_1 + A_0,
\quad A_i^T=-A_i, \ A_i \in \mathbb C^{m \times m} \ \text{for } i=0, \dots, d.
\end{equation*}
Recall that the degree of $P(\lambda)$, denoted as $\deg(P)$, is the largest index $k$ such that $A_k \ne 0$.
 By $\POL_{d, m}^{ss}$ we denote the vector space of $m \times m$ skew-symmetric matrix polynomials of grade $d$.
 If there is no risk of confusion, we write $\POL$ instead of $\POL_{d, m}^{ss}$. Using the Frobenius matrix norm of complex matrices \cite{High02}, a distance in $\POL_{d, m}^{ss}$ is defined as $d(P,P') = \left( \sum_{i=0}^d || A_i - A'_i ||_F^2 \right)^{\frac{1}{2}}$, where $P(\lambda) = \sum_{i=0}^d \lambda^i A_i$ and $P'(\lambda) = \sum_{i=0}^d \lambda^i A'_i$, making $\POL_{d, m}^{ss}$ a metric space  with the Euclidean topology induced by this distance. We define the Frobenius norm of the matrix polynomial $P$ as $||P(\lambda)||_F = \left( \sum_{i=0}^d || A_i ||_F^2 \right)^{\frac{1}{2}}$. The fact that $\POL_{d, m}^{ss}$ is a metric space allows us to consider the closure of any of its subsets $\mathcal{S}$, which is denoted by $\overline{\mathcal{S}}$.

Two matrix polynomials $P(\lambda)$ and $Q(\lambda)$ are called {\it unimodularly congruent} if $F(\lambda)^T P(\lambda) F(\lambda)=Q(\lambda)$ for some unimodular matrix polynomial $F(\lambda)$ (i.e. $\det F(\lambda) \in \mathbb C \backslash \{0\}$), see also \cite{MMMM13}.
\begin{theorem}{\rm \cite{MMMM13}} \label{tsmiths}
Let $P(\lambda)$ be a skew-symmetric $m\times m$ matrix polynomial. Then there exist $r \in \mathbb N$ with $2r \le m$ and a unimodular matrix polynomial $F(\lambda)$ such that
\begin{equation*}
F(\lambda)^T P(\lambda) F(\lambda) =
\begin{bmatrix}
0 & g_1(\lambda)\\
-g_1(\lambda) & 0
\end{bmatrix} \oplus
\dots
\oplus
\begin{bmatrix}
0 & g_r(\lambda)\\
-g_r(\lambda) & 0
\end{bmatrix} \oplus
0_{m-2r}=:S(\lambda),
\end{equation*}
where $g_j$ is a monic polynomial, for $j=1, \dots, r$, and $g_j(\lambda)$ divides $g_{j+1}(\lambda)$, for $j=1, \dots, r-1$. Moreover, the canonical form $S(\lambda)$ is unique.
\end{theorem}

The $2r$ monic scalar polynomials $g_1(\lambda), g_1(\lambda), g_2(\lambda), g_2(\lambda), \dots,g_r(\lambda), g_r(\lambda)$ in Theorem \ref{tsmiths} are called the {\it invariant polynomials} of $P(\lambda)$, and, for any $\alpha \in \mathbb{C}$, each of them can be uniquely factored as
$$
g_j(\lambda)  = (\lambda - \alpha)^{\sigma_j} p_j (\lambda), \quad \text{ with } p_j(\alpha) \ne 0
$$
and $\sigma_j \ge 0$ { being} an integer. The sequence $0 \leq \sigma_1 = \sigma_1 \leq \sigma_2 = \sigma_2 \leq \cdots \leq \sigma_r = \sigma_r$ is called {\it the sequence of partial multiplicities} of $P(\lambda)$ at $\alpha$. The number $\alpha$ is a finite eigenvalue of $P(\lambda)$ if the partial multiplicity sequence at $\alpha$ contains at least two nonzero terms, or, equivalently, if $\alpha$ is a root of at least one invariant polynomial $g_j (\lambda)$. The {\it elementary divisors} of $P(\lambda)$ associated with a finite eigenvalue $\alpha$ is the collection of factors $(\lambda - \alpha)^{\sigma_1},(\lambda - \alpha)^{\sigma_1}, (\lambda - \alpha)^{\sigma_2},(\lambda - \alpha)^{\sigma_2}, \ldots , (\lambda - \alpha)^{\sigma_r}, (\lambda - \alpha)^{\sigma_r}$ for which $\sigma_j > 0$.
The (normal) rank of the skew-symmetric matrix polynomial $P(\lambda)$ in Theorem \ref{tsmiths} is equal to the nonnegative integer $2r$, i.e., the normal rank is always even.
Note also that a square matrix polynomial $P(\lambda)$ is singular if and only if $\det(P(\lambda)) = 0$, for all $\lambda$.

For completion, we recall that $\lambda = \infty$ is an eigenvalue of the matrix polynomial $P(\lambda) \in \POL^{ss}_{d, m}$ if zero is an eigenvalue of $\rev_d P(\lambda):= \lambda^d P(1/\lambda)$. The elementary divisors and partial multiplicities for the zero eigenvalue of $\rev_d P(\lambda)$ are the elementary divisors and partial multiplicities associated with the infinite eigenvalue of $P(\lambda)$.


For an $m\times n$ matrix polynomial $P(\lambda)$, the left and right null-spaces, over the field of rational functions $\mathbb C(\lambda)$, are defined as follows:
\begin{align*}
{\cal N}_{\rm left}(P)&:= \{y(\lambda)^T \in \mathbb C(\lambda)^{1 \times m}: y(\lambda)^TP(\lambda) = 0_{1\times n} \}, \\
{\cal N}_{\rm right}(P)&:= \{x(\lambda) \in \mathbb C(\lambda)^{n\times 1}: P(\lambda)x(\lambda) = 0_{m\times 1}\}.
\end{align*}
Each subspace ${\cal V}$ of $\mathbb C(\lambda)^n$ has bases consisting entirely of vector polynomials.
A basis of ${\cal V}$ consisting of vector polynomials whose sum of degrees is minimal among all bases of ${\cal V}$ consisting of vector polynomials is called a {\it minimal basis} of ${\cal V}$. The ordered list of degrees of the vector polynomials in any minimal basis of ${\cal V}$ is always the same. These degrees are called the minimal indices of ${\cal V}$ \cite{Forn75,Kail80}. Thus we can define the left and right {\it minimal indices} of a matrix polynomial $P(\lambda)$ as those of ${\cal N}_{\rm left}(P)$ and ${\cal N}_{\rm right}(P)$, respectively. Note that for a skew-symmetric matrix polynomial the left minimal indices are equal to the right ones.

Following \cite{DeDD24,Dmyt17,DmDo18}, we define the {\it complete eigenstructure} of a skew-symmetric matrix polynomial $P(\lambda) \in \POL^{ss}_{d, m}$ {\it of grade $d$} to be the collection of all finite and infinite eigenvalues, the corresponding elementary divisors (or, equivalently, the corresponding sequences of partial multiplicities), and the left and right minimal indices of $P(\lambda)$. For convenience, we also define the notion of orbit of a skew-symmetric matrix polynomial $P(\lambda) \in \POL^{ss}_{d, m}$, denoted $\orb(P(\lambda))$, as the subset of matrix polynomials in $\POL^{ss}_{d, m}$ with the same complete eigenstructure as $P(\lambda)$. 
Observe that all the polynomials in $\orb (P)$ have the same {\it degree} and the same rank. 
Note that the definition of orbit for general polynomials is analogous and given in, e.g., \cite{DmDo17,DJKV20}.

The subset of matrix polynomials in $\POL^{ss}_{d, \, m}$ with normal rank at most $2r$ will be denoted by $\POL^{ss}_{d, \, 2r, \,  m}$, or simply by $\POL_{2r}$ when there is no need to specify the other parameters or they are clear from the context. A simple description of $\POL^{ss}_{d, \, 2r, \, m}$---as the closure of a certain “generic set” characterized by its eigenstructure---was obtained in \cite{DeDD24,DmDo18} and is recalled in Theorem~\ref{anydth}. Such descriptions are also known for matrix polynomials with other types of structures \cite{DeDD20b,DmDo17}. 
\begin{theorem} {\rm \cite{DeDD24}} \label{anydth}
Let $m,r$, and $d$ be integers such that $m > 2$, $d \geq 1$, and $2 \leq 2r \leq (m-1)$.
The set of $m\times m$ complex skew-symmetric matrix polynomials of grade $d$ with rank at most $2r$ is a closed subset of $\POL_{d, m}^{ss}$ equal to $\overline{\orb (W)}$, where { $W \in \POL^{ss}_{d, m}$} is an $m \times m$ complex skew-symmetric matrix polynomial of degree exactly $d$ and rank exactly $2r$ with no elementary divisors at all, with $t$ left minimal indices equal to $(\beta +1)$ and with $(m-2r-t)$ left minimal indices equal to $\beta$, where $\beta = \lfloor rd / (m-2r) \rfloor$ and $t = rd \mod (m-2r)$, and with the right minimal indices equal to the left minimal indices.
\end{theorem} 

Theorem \ref{anydth} is one of the main tools for obtaining another simple description of $\POL^{ss}_{d, \, 2r, \,  m}$ as the closure of a certain “generic set” of skew-symmetric matrix polynomials characterized by the simple way in which they can be factorized. This simple description is obtained in Section \ref{sec.genfacts} and is the key result on which the algorithm we propose is based on.

\section{Generic factorizations of skew-symmetric matrix polynomials with bounded rank} \label{sec.genfacts}
Theorem \ref{thm.mainskew} is a key result for the development of the algorithm presented in this paper. It is one of the results included in \cite{DmDV26}. Although the statement of Theorem \ref{thm.mainskew} is very simple, its proof is long, relies on several preliminary lemmas, and is based on results in \cite{DmDV23}. Therefore, the proof is included in Subsection \ref{subsec.proofmain}, which can be ommitted without affecting the understanding of the rest of the paper. In Theorem \ref{thm.mainskew}, as well as in other results, some general matrix polynomials (i.e., not necessarily skew-symmetric) are used, which motivates the introduction of the following notation: $\POL^{p\times q}_d$ is the vector space of $p\times q$ matrix polynomials of grade $d$ over $\mathbb{C}$.

\begin{theorem} \label{thm.mainskew} Let $m,r$ and $d$ be integers such that $d\geq 1$ and $2 \leq 2r \leq (m-1)$. Let us define the following subset of $\POL^{ss}_{d, \, 2r, \,  m}$
$$
\mathcal{G}^{ss}_{d, \, 2r, \, m}  :=  \left\{U(\lambda) V(\lambda)^T - V(\lambda) U(\lambda)^T  \, : \,
\begin{array}{l} \displaystyle
V(\lambda) \in \POL^{m\times r}_{\displaystyle \left\lceil d/2\right\rceil}, \\[0.4 cm] \displaystyle
U(\lambda) \in \POL^{m\times r}_{\displaystyle \left\lfloor d/2\right\rfloor}
\end{array}
\right\}  \, .
$$
Then,
$$
\overline{\mathcal{G}^{ss}_{d, \, 2r, \, m}} = \POL^{ss}_{d, \, 2r, \,  m}.
$$
\end{theorem}

\subsection{Proof of Theorem \ref{thm.mainskew}} \label{subsec.proofmain}
The proof uses several definitions and results that can be found in \cite[Sections 2 \& 3]{DmDV23}, which are not repeated here for the sake of brevity. We use often the $i$-th column of a matrix polynomial $L(\lambda)$, which is denoted by $L_{*i} (\lambda)$ of simply by $L_{*i}$. In addition to the notations used above for different sets of matrix polynomials, we use in this subsection $\POL^{p\times q}$ to denote the set of $p\times q$ matrix polynomials over $\mathbb{C}$, without specifying the grade. The (normal) rank of a matrix polynomial $L(\lambda)$ is denoted by $\rank (L)$. For $k = 1,2,\ldots$, we define the following $k \times k$ matrix
$$
\Delta_k := \begin{bmatrix}
               &  &  & 1 \\
               &   & 1 &   \\
              & \iddots &  &  \\
             1 &  &  &  
           \end{bmatrix}.
$$
With this definition, it can be shown that the set $\mathcal{G}^{ss}_{d, \, 2r, \, m}$ in Theorem \ref{thm.mainskew} can be expressed as follows
\begin{equation} \label{eq.Gset2}
\mathcal{G}^{ss}_{d, \, 2r, \, m}  =  \left\{ L(\lambda) \left[ \begin{array}{cc}
0 & \Delta_r \\
-\Delta_r & 0
\end{array}
\right]  L(\lambda)^T\, : \,
\begin{array}{l}
L(\lambda) \in \POL^{m\times 2r}, \\ \displaystyle
\deg (L_{*i}) \leq \left\lceil \frac{d}{2}\right\rceil, \quad \mbox{for $i=1,\ldots , r$}, \\[5mm] \displaystyle
\deg (L_{*i}) \leq \left\lfloor \frac{d}{2}\right\rfloor, \quad \mbox{for $i=r + 1,\ldots ,2r$}
\end{array}
\right\}  \, ,
\end{equation}
which is more adequate for the purposes of this subsection. To see that the set in the right-hand side of \eqref{eq.Gset2} is indeed $\mathcal{G}^{ss}_{d, \, 2r, \, m}$, it is enough to express $L(\lambda)$ as $L(\lambda) = \begin{bmatrix}
-V(\lambda)\Delta_r & U(\lambda)
\end{bmatrix}$.

If $P(\lambda) \in \POL^{ss}_{d, m}$, then its rational column space, as defined in \cite[Definition 2.6]{DmDV23}, satisfies ${\mathcal Col} (P) = {\mathcal Col} (P^T)$. Lemma \ref{lemm.pr31-1} is the first auxiliary result we need for proving Theorem \ref{thm.mainskew}. It is a consequence of \cite[Theorems 3.11 and 3.19]{DmDV23} and uses the well-known concept of {\em minimal basis} of a rational subspace \cite{Forn75} (see also \cite[pp. 742-743]{DmDV23}). Following a standard practice, we often say that a matrix polynomial is a minimal basis when its columns are a minimal basis of the rational subspace they span.

\begin{lemma} \label{lemm.pr31-1}
\begin{itemize}
    \item[\rm (i)] If $P(\lambda) \in \POL^{ss}_{d, 2r, m}$ and $\rank (P) = 2r$, then $P(\lambda) = L(\lambda) E(\lambda)  L(\lambda)^T$, where the columns of $L(\lambda) \in \POL^{m \times 2r}$ form a minimal basis of  ${\mathcal Col} (P)$ and $E(\lambda) \in \POL^{2r \times 2r}$ is skew-symmetric with $\rank (E) = 2r$.
    \item[\rm (ii)] If $P(\lambda) \in \POL^{ss}_{d, 2r, m}$, $\rank (P) = 2r$, $\deg(P) = d$, $P(\lambda)$ has not eigenvalues (finite or infinite) and the columns of $L(\lambda)$ in the factorization in item {\rm (i)} are ordered in such a way that $\deg (L_{*1}) \geq \deg (L_{*2}) \geq \cdots \geq \deg (L_{*2r})$, then 
    $$
    \deg (L_{*i}) + \deg (L_{*(2r-i+1)}) = d, \quad \mbox{for $i = 1, \ldots , 2r$}.
    $$
\end{itemize}
\end{lemma}
\begin{proof}
Item (i). Except for the skew-symmetry of $E(\lambda)$, the existence of a factorization $P(\lambda) = L(\lambda) E(\lambda)  L(\lambda)^T$ with the stated properties follows immediately from combining Theorem 3.11-(i) and Remark 3.13 in \cite{DmDV23} with the fact that ${\mathcal Col} (P) = {\mathcal Col} (P^T)$. For proving the skew-symmetry of $E(\lambda)$, we use that since $L(\lambda)$ is a minimal basis, it has a polynomial left inverse $M(\lambda) \in \POL^{2r \times m}$ such that $M(\lambda) L(\lambda) = I_{2r}$ \cite[Theorem 4]{Forn75}. So, $E(\lambda) = M(\lambda) P(\lambda) M(\lambda)^T$ and the skew-symmetry of $E(\lambda)$ follows from that of $P(\lambda)$.

\medskip \noindent
Item (ii). Theorem 3.19 in \cite{DmDV23} implies that $P(\lambda) = L(\lambda) U(\lambda)$, where
$L(\lambda), U(\lambda)^T \in  \POL^{m \times 2r}$ are two minimal bases of ${\mathcal Col} (P)$ and
\begin{equation} \label{eq.aux31-1}
    \deg (L_{*i}) + \deg ((U^T)_{*i}) = d, \quad \mbox{for $i = 1, \ldots , 2r$}.
\end{equation}
Since, $L(\lambda)$ and $U(\lambda)^T$ are minimal bases of the same rational subspace, there exists a reordering of the columns of $U(\lambda)^T$ that makes the degrees of its columns equal to the degrees of the columns of $L(\lambda)$. Moreover, \eqref{eq.aux31-1} guarantees that if the degrees of the columns of $L(\lambda)$ are decreasingly ordered, then those of the columns of $U(\lambda)^T$ are increasingly ordered.
\end{proof}

Motivated by the properties of the matrix polynomials in the orbit $\orb (W)$ in Theorem \ref{anydth}, we introduce the following subset of $\POL^{ss}_{d, \, 2r, \,  m}$, which includes $\orb (W)$.

\begin{definition} \label{def.polynoeig}
$$
\POL^{ss,neig}_{d, \, 2r, \,  m} := \left\{
P(\lambda) \in \POL^{ss}_{d, \, 2r, \,  m} \, : \,
\begin{array}{l}
\rank(P) = 2r, \\
\deg(P) = d, \\
P(\lambda) \; \mbox{has no eigenvalues}
\end{array}
\right\} \, .
$$
\end{definition}
We emphasize that when we say that a matrix polynomial has no eigenvalues, we refer to both finite and infinite eigenvalues.

The next lemma is a corollary of Theorem \ref{anydth} and of $\orb (W) \subset \POL^{ss,neig}_{d, \, 2r, \,  m} \subset \POL^{ss}_{d, \, 2r, \,  m}$. One just need to take closures in this chain of inclusions. Note that Theorem \ref{anydth} states, in particular, that $\POL^{ss}_{d, \, 2r, \,  m}$  is a closed subset of $\POL^{ss}_{d, \, m}$, which, on the other hand, is obvious from the definition of (normal) rank. Lemma \ref{lem.1incl} is important in the proof of Theorem \ref{thm.mainskew}.

\begin{lemma} \label{lem.1incl}
$\overline{\POL^{ss,neig}_{d, \, 2r, \,  m}} = \POL^{ss}_{d, \, 2r, \,  m}$.
\end{lemma}

Theorem \ref{thm.auxmainskew} describes the set $\POL^{ss,neig}_{d, \, 2r, \,  m}$ in terms of minimal rank factorizations.
\begin{theorem} \label{thm.auxmainskew} 
\[
\POL^{ss,neig}_{d, \, 2r, \,  m} =  \left\{ L(\lambda) \left[ \begin{array}{cc}
0 & \Delta_r \\
-\Delta_r & 0
\end{array}
\right]  L(\lambda)^T\, : \,
\begin{array}{l}
L(\lambda) \in \POL^{m\times 2r}, \\
L(\lambda) \; \mbox{is a minimal basis}, \\ 
\deg(L_{*1}) \geq \cdots \geq \deg(L_{*2r}), \\
\deg(L_{*i}) + \deg(L_{*(2r-i+1)}) = d, \\
\phantom{aa} \mbox{for $i=1,\ldots ,2r$}
\end{array}
\right\}  \, .
\]
\end{theorem}
\begin{proof}
Let $\mathcal{S}$ be the set in the right-hand side of the equality in Theorem \ref{thm.auxmainskew}. The inclusion $\mathcal{S}\subseteq \POL^{ss,neig}_{d, \, 2r, \,  m}$ is a direct consequence of \cite[Theorem 3.19]{DmDV23}. Therefore, we focus on proving that if $P(\lambda) \in \POL^{ss,neig}_{d, \, 2r, \,  m}$, then $P(\lambda) \in \mathcal{S}$. By Lemma \ref{lemm.pr31-1}, any $P(\lambda) \in \POL^{ss,neig}_{d, \, 2r, \,  m}$ can be factorized as 
\begin{equation} \label{eq.tildefact}
P(\lambda) = \widetilde L(\lambda) \widetilde E(\lambda)  \widetilde L(\lambda)^T,
\end{equation}
where the columns of $\widetilde L(\lambda) \in \POL^{m \times 2r}$ are a minimal basis of  ${\mathcal Col} (P)$ and $\widetilde E(\lambda) \in \POL^{2r \times 2r}$ is skew-symmetric with $\rank (\widetilde E) = 2r$. Moreover, $\deg (\widetilde L_{*1}) \geq \deg (\widetilde L_{*2}) \geq \cdots \geq \deg (\widetilde L_{*2r})$, and
\begin{equation} \label{eq.reldeg1}
\deg (\widetilde L_{*i}) + \deg (\widetilde L_{*(2r-i+1)}) = d, \quad \mbox{for $i = 1, \ldots , 2r$}.
\end{equation}
Let us assume that the first $r$ columns of $\widetilde L(\lambda)$ have $p$ different degrees $\widetilde{d}_1 > \cdots > \widetilde{d}_p$ and that there are $r_i$ columns with degree $\widetilde{d}_i$ for $i=1, \ldots , p$. Note that $r = r_1 + \cdots + r_p$. From \eqref{eq.reldeg1}, we get that the last $r$ columns of $\widetilde L(\lambda)$ also have $p$ different degrees $\widetilde{d}_{p+1} > \cdots > \widetilde{d}_{2p}$ and that there are $r_i$ columns among the last $r$ columns with degree 
\begin{equation} \label{eq.reldeg2}
\widetilde{d}_{2p-i+1} = d - \widetilde{d}_{i}, \quad \mbox{for $i=1, \ldots , p$},
\end{equation}
which are in the ``mirror'' positions to the corresponding ones in the first $r$ columns.
These considerations induce the following partition of  $\widetilde L(\lambda)$
\begin{equation} \label{eq.partitiontildeL}
\widetilde L(\lambda) = \begin{bmatrix}
\widetilde L_1 & \cdots & \widetilde L_p & \widetilde L_{p+1}  & \cdots & \widetilde L_{2p}
\end{bmatrix},
\end{equation}
where we omit the dependence in $\lambda$ in the blocks $\widetilde L_i$,
\begin{equation} \label{eq.propstildeL}
\begin{array}{l}
\widetilde L_i \in \POL^{m \times r_i}, \; \mbox{for $i=1, \ldots , p$},\\[2mm]
\widetilde L_i \in \POL^{m \times r_{2p-i+1}}, \; \mbox{for $i=p+1, \ldots , 2p$},\\[2mm]
\deg (\widetilde L_i) = \widetilde{d}_i, \; \mbox{for $i=1, \ldots , 2p$}, \\[2mm]
\widetilde{d}_1 > \cdots > \widetilde{d}_p \geq \widetilde{d}_{p+1} > \cdots > \widetilde{d}_{2p},
\end{array} 
\end{equation}
and all the columns in the same block $\widetilde L_i$ have the same degree $\widetilde{d}_i$.
Note that $\widetilde{d}_p \geq d/2$ and $\widetilde{d}_{p+1} \leq d/2$, because $\widetilde{d}_p \geq \widetilde{d}_{p+1}$ and $\widetilde{d}_p + \widetilde{d}_{p+1} = d$ by \eqref{eq.reldeg2}. So,
\begin{equation} \label{eq.reldeg3}
\widetilde{d}_1 > \cdots > \widetilde{d}_p \geq d/2 \geq \widetilde{d}_{p+1} > \cdots > \widetilde{d}_{2p}.
\end{equation}
The partition of $\widetilde L(\lambda)$ induces a conformable partition in the skew-symmetric matrix polynomial $\widetilde E (\lambda)$. By \cite[Corollary 3.17-(ii)]{DmDV23} such a partition satisfies the following
\begin{equation} \label{eq.reldeg4}
\widetilde E (\lambda) = \left( \widetilde{E}_{ij} \right)_{1\leq i,j \leq 2p} \; \mbox{with} \quad
\widetilde{E}_{ij} = - \widetilde{E}_{ji}^T , \quad \deg(\widetilde{E}_{ij}) \leq d - \widetilde{d}_i  - \widetilde{d}_j, 
\end{equation}
and where we omit the dependence in $\lambda$ in the blocks $\widetilde E_{ij}$. 

The next step in the proof is to show that $\widetilde E (\lambda)$ can be chosen to be lower block anti-triangular, i.e., $\widetilde{E}_{ij}=0$ for $i+j \leq 2p$. For that purpose, we consider two cases. First, if $\widetilde{d}_p >\widetilde{d}_{p+1}$, then \eqref{eq.reldeg3} becomes $\widetilde{d}_1 > \cdots > \widetilde{d}_p > d/2 > \widetilde{d}_{p+1} > \cdots > \widetilde{d}_{2p}$, which, combined with the degree bound in \eqref{eq.reldeg4} and \eqref{eq.reldeg2}, implies that $\widetilde E (\lambda)$ is indeed lower block anti-triangular.  
Second, if $\widetilde{d}_p =d/2 =\widetilde{d}_{p+1}$, then the degree bound in \eqref{eq.reldeg4} and \eqref{eq.reldeg2} only guarantee $\widetilde{E}_{ij}=0$ for $i+j \leq 2p$ and $(i,j) \ne (p,p)$. In this case, the skew-symmetric submatrix $\left( \widetilde{E}_{ij} \right)_{p\leq i,j \leq p+1}$ containing the only potential non-zero block above the main block anti-diagonal must be invertible, since $\rank(\widetilde{E}) = 2r$, and must have degree zero by the degree bound in \eqref{eq.reldeg4} and $\widetilde{d}_p =d/2 =\widetilde{d}_{p+1}$. That is, it is a constant invertible matrix. Therefore by \cite[Corollary 2.6.6-(b)]{HoJo2013}, there exists an invertible constant complex matrix $Q$ such that
$$
Q \begin{bmatrix}
\widetilde E_{p,p} & \widetilde  E_{p, p+1} \\
\widetilde  E_{p+1 , p} & \widetilde  E_{p+1 , p+1}
\end{bmatrix} Q^T = \begin{bmatrix}
0 & H \\
-H^T & 0
\end{bmatrix},
$$
with $H \in \mathbb{C}^{r_p \times r_p}$ invertible. Define $\widetilde{Q} = \diag (I_\ell,Q,I_\ell)$, where $\ell = r-r_p$, and use \eqref{eq.tildefact} to express $P(\lambda)$ as
\[
P(\lambda) = (\widetilde L(\lambda) \widetilde{Q}^{-1}) \, (\widetilde{Q} \widetilde E(\lambda)  \widetilde{Q}^T) \, (\widetilde L(\lambda) \widetilde{Q}^{-1})^T.
\]
Observe that $(\widetilde{Q} \widetilde E(\lambda)  \widetilde{Q}^T)$ is now skew-symmetric, lower block anti-triangular, and the properties of its blocks are still those stated in \eqref{eq.reldeg4}. On the other hand, the polynomial matrix $(\widetilde L(\lambda) \widetilde{Q}^{-1})$ is a minimal basis of ${\mathcal Col} (P)$, because has full column rank $2r$ and the degrees of their columns are the same as those of $\widetilde L(\lambda)$, and admits a partition as the one in \eqref{eq.partitiontildeL} with the properties stated in \eqref{eq.propstildeL}. To avoid additional notation, we actualize the definitions of $\widetilde L(\lambda) \leftmapsto (\widetilde L(\lambda) \widetilde{Q}^{-1})$ and of $\widetilde E(\lambda) \leftmapsto (\widetilde{Q} \widetilde E(\lambda)  \widetilde{Q}^T)$. That is, in the rest of the proof, we work with a factorization of $P(\lambda)$ as in \eqref{eq.tildefact} assuming that $\widetilde E (\lambda)$ is block lower anti-triangular. Observe that the blocks $\widetilde{E}_{i,2p-i+1}$, $i=1,\ldots, 2p$, on the main block anti-diagonal must be constant matrices, due to \eqref{eq.reldeg2} and \eqref{eq.reldeg4}, and invertible, because $\rank (\widetilde{E}) = 2r$.

The lower block anti-triangularity and the skew-symmetry of $\widetilde{E} (\lambda)$ allow us to establish, through a direct multiplication, the following factorization 
\begin{equation} \label{eq.factEtil1}
\widetilde{E} (\lambda) = X(\lambda) \, \Lambda \, X(\lambda)^T,
\end{equation}
where
\[
X(\lambda) :=
\begin{bmatrix}
\widetilde{E}_{1,2p}  \Delta_{r_1}  & & & & & & &  \\
\widetilde{E}_{2,2p}   \Delta_{r_1}  & \widetilde{E}_{2,2p-1} \Delta_{r_2} & & & & & &  \\
\widetilde{E}_{3,2p}   \Delta_{r_1} & \widetilde{E}_{3,2p-1} \Delta_{r_2} & \ddots & & & & &  \\
\vdots   &\vdots & & \widetilde{E}_{p,p+1} \Delta_{r_p} & & & &  \\
\vdots   & \vdots & & \frac{\displaystyle \widetilde{E}_{p+1,p+1}}{2} \Delta_{r_p} & I_{r_p} &  & &  \\
\widetilde{E}_{2p-2,2p}  \Delta_{r_1}& \widetilde{E}_{2p-2,2p-1} \Delta_{r_2} & \reflectbox{$\ddots$} &  & & \ddots & &  \\
\widetilde{E}_{2p-1,2p}  \Delta_{r_1}  & \frac{\displaystyle \widetilde{E}_{2p-1,2p-1}}{2} \Delta_{r_2} & & & & & I_{r_2} &  \\
\frac{\displaystyle \widetilde{E}_{2p,2p}}{2}  \Delta_{r_1}  & & & & & & & I_{r_1}
\end{bmatrix}
\]
and
\[
\Lambda := \begin{bmatrix}
 & & & & & \Delta_{r_1} \\
 & & & & \iddots & \\
 & & &\Delta_{r_p} & & \\
 & & -\Delta_{r_p} & & & \\
 & \iddots & & & & \\
-\Delta_{r_1} & & & & & 
\end{bmatrix} = \left[ \begin{array}{cc}
0 & \Delta_r \\
-\Delta_r & 0
\end{array}
\right]  \in \mathbb{C}^{2r \times 2r} .
\]
Note that the diagonal blocks of $X(\lambda)$ are constant invertible matrices and, so, the matrix polynomial $X(\lambda)$ is invertible (in fact, it is unimodular).
Combining the factorizations \eqref{eq.tildefact} and \eqref{eq.factEtil1}, we get
\begin{equation} \label{eq.deffactor}
 P(\lambda) = (\widetilde{L} (\lambda)  X(\lambda)) \,  \left[ \begin{array}{cc}
0 & \Delta_r \\
-\Delta_r & 0
\end{array}
\right] \, (\widetilde{L}(\lambda)  X(\lambda))^T,
\end{equation}
which imply that the columns of $L(\lambda) := \widetilde{L} (\lambda)  X(\lambda)$ form a polynomial basis of ${\mathcal Col} (P)$.
Taking into account \eqref{eq.reldeg2}, \eqref{eq.partitiontildeL}, \eqref{eq.propstildeL}, and \eqref{eq.reldeg4}, we can partition $L(\lambda)$ as
\begin{equation} \label{eq.partitionLfin}
L(\lambda) = \begin{bmatrix}
L_1 & \cdots & L_p & L_{p+1}  & \cdots &  L_{2p}
\end{bmatrix},
\end{equation}
where $L_j = \widetilde{L}_j$ for $j = p+1, \ldots , 2p$, and $\deg (L_j) \leq  \deg (\widetilde{L}_j) = \widetilde{d}_j$ for $j = 1, \ldots , p$. But if some column of some block $L_j$ had degree strictly smaller than $\widetilde{d}_j$, we would have a polynomial basis of ${\mathcal Col} (P)$ with the sum of all the degrees of its vectors strictly smaller than the sum of all the degrees of the column vectors of the minimal basis $\widetilde{L} (\lambda)$ of ${\mathcal Col} (P)$, which is impossible. So, $L(\lambda)$ is a minimal basis with the degrees of its columns exactly equal to those of $\widetilde L(\lambda)$, which concludes the proof.
\end{proof}

The factorizations in Theorem \ref{thm.auxmainskew} satisfy 
$$\deg(L_{*1}) \geq \cdots \geq \deg(L_{*r}) \geq d/2 \geq \deg(L_{*(r+1)}) \geq \cdots \geq
\deg(L_{*2r}),$$
because otherwise $\deg(L_{*r}) + \deg(L_{*(r+1)}) = d$ would not hold. This motivates the following definition of certain subsets of $\POL^{ss}_{d, \, 2r, \,  m}$.

\begin{definition} \label{def.dsets} Let $m, r$ and $d$ be integers such that $d\geq 1$ and $2 \leq 2r \leq (m-1)$. Let $d_1, d_2, \ldots, d_r$ be integers such that $d \geq d_i \geq d/2$ for $i=1, \ldots, r$. Then
$$
\mathcal{A} (d_1, \ldots , d_r) := 
\left\{ L(\lambda) \left[ \begin{array}{cc}
0 & \Delta_r \\
-\Delta_r & 0
\end{array}
\right]  L(\lambda)^T\, : \,
\begin{array}{l}
L(\lambda) \in \POL^{m\times 2r}, \\
\deg(L_{*i}) = d_i,  \\
\deg(L_{*(2r-i+1)}) = d - d_i, \\
\phantom{aa} \mbox{for $i=1,\ldots ,r$}
\end{array}
\right\}  \, .
$$   
\end{definition}
If Definition \ref{def.dsets} is compared with Theorem \ref{thm.auxmainskew}, we see that in the matrix polynomials in $\mathcal{A} (d_1, \ldots , d_r)$, the factor $L(\lambda)$ is not required to be a minimal basis and that the columns of $L(\lambda)$ are not necessarily arranged in decreasing degree order. On the other hand, it is obvious that $\mathcal{A} (d_1, \ldots , d_r) \subset \POL^{ss}_{d, \, 2r, \,  m}$ and that any polynomial in $\POL^{ss,neig}_{d, \, 2r, \,  m}$ belongs to one of these sets  $\mathcal{A} (d_1, \ldots , d_r)$.

The key property of the sets in Definition \ref{def.dsets} in the context of this paper is stated in the next theorem. We observe that if in the description of the set $\mathcal{G}^{ss}_{d, \, 2r, \, m}$ in \eqref{eq.Gset2} the degree inequalities were replaced by strict equalities, then we would obtain the set $\mathcal{A} \left( \left\lceil \frac{d}{2} \right\rceil, \ldots , \left\lceil \frac{d}{2} \right\rceil \right)$ that appears in Theorem \ref{thm.dsets}.

\begin{theorem} \label{thm.dsets} If $d_1, d_2, \ldots, d_r$ are any integers such that $d \geq d_i \geq d/2$ for $i=1, \ldots, r$, then
$$
\mathcal{A} (d_1, \ldots , d_r) \subseteq \overline{\mathcal{A} \left( \left\lceil \frac{d}{2} \right\rceil, \ldots , \left\lceil \frac{d}{2} \right\rceil \right)} .
$$
\end{theorem}
\begin{proof}
Let us assume that there exists one integer $d_j > \left\lceil \frac{d}{2} \right\rceil$. Otherwise, there is nothing to prove. Along the proof, for any index $i = 1, \ldots , r$, we define $i' := 2r - i +1$ and $d_{i'} := d -d_i$. Let $P (\lambda) \in \mathcal{A} (d_1, \ldots , d_r)$. Then,
$$
P(\lambda) = \sum_{i=1}^r P_i (\lambda), \quad \mbox{with} \; \; P_i (\lambda) := \begin{bmatrix}
L_{*i} & L_{*i'}
\end{bmatrix}
\begin{bmatrix}
0 & 1 \\-1 & 0
\end{bmatrix}
\begin{bmatrix}
L_{*i}^T \\ L_{*i'}^T
\end{bmatrix},
$$
where we omit the dependence in $\lambda$ in the columns of $L(\lambda)$. Let us focus on the summand $P_j (\lambda)$ corresponding to $d_j > \left\lceil \frac{d}{2} \right\rceil$, which implies $d_{j'} < \left\lfloor \frac{d}{2} \right\rfloor$ and $d_j - d_{j'} \geq 2$. Then
\[
L_{*j}  = v_j \lambda^{d_j}  + v(\lambda) \quad \mbox{and} \quad
L_{*j'}  = w_{j'} \lambda^{d_{j'}} + w(\lambda),
\]
where $v_j, w_{j'} \in \mathbb{C}^m$ are nonzero constant vectors, and $v(\lambda)$ and $w(\lambda)$ are polynomial vectors such that $\deg(v(\lambda)) < d_j$ and $\deg (w(\lambda)) < d_{j'}$. For $\varepsilon >0$ define
$$
L_{*j'}^{(\varepsilon)}  := \varepsilon \lambda^{d_{j'}+1}  v_j + w_{j'} \lambda^{d_{j'}} + w(\lambda) \;\; \mbox{and} \;\;
P_j^{(\varepsilon)} (\lambda) := \begin{bmatrix}
L_{*j} & L_{*j'}^{(\varepsilon)}
\end{bmatrix}
\begin{bmatrix}
0 & 1 \\-1 & 0
\end{bmatrix}
\begin{bmatrix}
L_{*j}^T \\ (L_{*j'}^{(\varepsilon)})^T
\end{bmatrix}.
$$
Observe that $\lim_{\varepsilon \rightarrow 0} P_j^{(\varepsilon)} (\lambda) = P_j (\lambda)$.
Since
$$
\begin{bmatrix}
0 & 1 \\-1 & 0
\end{bmatrix} = 
\begin{bmatrix}
1 & 0 \\-\frac{\lambda^{d_j - d_{j'} -1}}{\varepsilon}  & 1
\end{bmatrix}
\begin{bmatrix}
0 & 1 \\-1 & 0
\end{bmatrix}
\begin{bmatrix}
1 & 0 \\-\frac{\lambda^{d_j - d_{j'}-1}}{\varepsilon}  & 1
\end{bmatrix}^T \, ,
$$
we have that
$$
P_j^{(\varepsilon)} (\lambda) := \left(\begin{bmatrix}
L_{*j} & L_{*j'}^{(\varepsilon)}
\end{bmatrix}
\begin{bmatrix}
1 & 0 \\-\frac{\lambda^{d_j - d_{j'} -1}}{\varepsilon}  & 1
\end{bmatrix} \right)
\begin{bmatrix}
0 & 1 \\-1 & 0
\end{bmatrix}
\left(\begin{bmatrix}
L_{*j} & L_{*j'}^{(\varepsilon)}
\end{bmatrix}
\begin{bmatrix}
1 & 0 \\-\frac{ \lambda^{d_j - d_{j'} -1}}{\varepsilon} & 1
\end{bmatrix} \right)^T.
$$
Note that the columns of
$$
\begin{bmatrix}
L_{*j} & L_{*j'}^{(\varepsilon)}
\end{bmatrix}
\begin{bmatrix}
1 & 0 \\-\frac{\lambda^{d_j - d_{j'} -1}}{\varepsilon}  & 1
\end{bmatrix}  =  \begin{bmatrix}
v(\lambda)-  w_{j'} \frac{\lambda^{d_{j}-1}}{\varepsilon} - w(\lambda) \frac{ \lambda^{d_j - d_{j'} -1}}{\varepsilon} & L_{*j'}^{(\varepsilon)}
\end{bmatrix}
$$
have degrees $d_{j}-1$ and $d_{j'}+1$, respectively, if $\varepsilon$ is sufficiently small. Then
$$
P^{(\varepsilon)}(\lambda) := P_j^{(\varepsilon)} (\lambda) + \sum_{i=1, i \ne j}^r P_i (\lambda) \in \mathcal{A} (d_1, \ldots, d_{j-1}, d_{j}-1 , d_{j+1}, \ldots , d_r)
$$
for all $\varepsilon$ sufficiently small. This, together with $\lim_{\varepsilon \rightarrow 0} P^{(\varepsilon)} (\lambda) = P(\lambda)$, implies $P (\lambda) \in \overline{\mathcal{A} (d_1, \ldots, d_{j-1}, d_{j}-1 , d_{j+1}, \ldots , d_r)}$ and 
$$
\mathcal{A} (d_1, \ldots, d_{j} , \ldots , d_r) \subseteq \overline{\mathcal{A} (d_1, \ldots, d_{j-1}, d_{j}-1 , d_{j+1}, \ldots , d_r)}.
$$
If $d_1= \cdots = d_{j-1} = d_{j}-1 = d_{j+1} = \cdots = d_r = \left\lceil \frac{d}{2} \right\rceil$, we are done. If not, we repeat the process above with the set $\mathcal{A} (d_1, \ldots, d_{j-1}, d_{j}-1 , d_{j+1}, \ldots , d_r)$.
\end{proof}

We are now in the position to prove Theorem \ref{thm.finalproof}, which includes the result in Theorem \ref{thm.mainskew}.
\begin{theorem} \label{thm.finalproof}
$$
 \overline{\mathcal{A} \left( \left\lceil \frac{d}{2} \right\rceil, \ldots , \left\lceil \frac{d}{2} \right\rceil \right)}  =  \overline{\mathcal{G}^{ss}_{d, \, 2r, \, m}} =  \POL^{ss}_{d, \, 2r, \,  m}.
$$
\end{theorem}
\begin{proof}
Note first that, since $\POL^{ss}_{d, \, 2r, \,  m}$ is a closed subset of $\POL^{ss}_{d, \,   m}$ and $\mathcal{A} \left( \left\lceil \frac{d}{2} \right\rceil, \ldots , \left\lceil \frac{d}{2} \right\rceil \right) \subseteq \POL^{ss}_{d, \, 2r, \,  m}$, we have that $\overline{\mathcal{A} \left( \left\lceil \frac{d}{2} \right\rceil, \ldots , \left\lceil \frac{d}{2} \right\rceil \right)} \subseteq \POL^{ss}_{d, \, 2r, \,  m}$. On the other hand, if 
$P(\lambda) \in \POL^{ss,neig}_{d, \, 2r, \,  m}$, then Theorem \ref{thm.auxmainskew} implies that
$$P(\lambda) = L(\lambda) \left[ \begin{array}{cc}
0 & \Delta_r \\
-\Delta_r & 0
\end{array}
\right]  L(\lambda)^T \in \mathcal{A} (\deg(L_{*1}), \ldots , \deg(L_{*r})).$$ Therefore,
$\POL^{ss,neig}_{d, \, 2r, \,  m} \subseteq \overline{\mathcal{A} \left( \left\lceil \frac{d}{2} \right\rceil, \ldots , \left\lceil \frac{d}{2} \right\rceil \right)} \subseteq \POL^{ss}_{d, \, 2r, \,  m}$ follows from Theorem \ref{thm.dsets}, and the equality $\overline{\mathcal{A} \left( \left\lceil \frac{d}{2} \right\rceil, \ldots , \left\lceil \frac{d}{2} \right\rceil \right)} = \POL^{ss}_{d, \, 2r, \,  m}$ from Lemma \ref{lem.1incl}. Finally, \eqref{eq.Gset2} and Definition \ref{def.dsets} yield $\mathcal{A} \left( \left\lceil \frac{d}{2} \right\rceil, \ldots , \left\lceil \frac{d}{2} \right\rceil \right) \subseteq \mathcal{G}^{ss}_{d, \, 2r, \, m} \subseteq \POL^{ss}_{d, \, 2r, \,  m}$. Taking the closures in this chain of inclusions completes the proof.
\end{proof}

\section{Distance to rank-deficient skew-symmetric matrix polynomials}
\label{sec:poldist}

Let $P(\lambda) = \sum_{i= 0}^d \lambda^iP_i,\text{ } P_i \in \mathbb{C}^{m \times m}, \text{ } i =0,...,d$, be a given matrix polynomial of degree $d$. Our goal is to find a matrix polynomial in $\POL^{ss}_{d, \, 2r, \, m}=\POL_{2r}$ that is as near as possible to $P(\lambda)$, i.e., to minimize the distance from $P(\lambda)$ to $\POL_{2r}$, 
\begin{equation} \label{pro:minp}
\dist_{\POL_{2r}}\left( \sum_{i= 0}^d \lambda^iP_i\right) = \min_{\sum_{i= 0}^d \lambda^iS_i \in \POL_{2r}} \left( \sum_{i= 0}^d \ || P_i -S_i ||^2_F \right)^{1/2}.
\end{equation}
This problem includes the well-known problem of finding the {\it distance to singularity}\footnote{ Recall that any skew-symmetric matrix polynomial of odd size is automatically singular. Therefore, if $m$ is odd the nearest singular skew-symmetric matrix polynomial to $P(\lambda)$ is just the nearest skew-symmetric matrix polynomial, which can be easily found by finding the nearest skew-symmetric matrix in the Frobenius norm to each of the matrix coefficients of $P(\lambda)$.} when $m$ is even
\begin{equation} \label{pro:distsing}
\dist_{\POL_{m-2}}\left( P \right) = \dist_{sing}\left( P \right) := \min \left\{ ||\Delta P||_F: (P+\Delta P)(\lambda) \text{ is singular} \right\}.
\end{equation}

%
Using Theorem \ref{thm.mainskew} we can substitute the minimization problem \eqref{pro:minp} with the following minimization problem (for brevity, we write $\mathcal{G}_{2r}$ instead of $\mathcal{G}^{ss}_{d, \, 2r, \, m }$)

\begin{equation} \label{pro:minc}
\begin{aligned}
&\dist_{\POL_{2r}}\left(\sum_{i= 0}^d \lambda^iP_i\right) = \min_{\sum_{i= 0}^d \lambda^iS_i \in \mathcal{G}_{2r}} \left( \sum_{i= 0}^d \ || P_i -S_i ||^2_F \right)^{1/2} \\ 
&= \min_{\sum_{i= 0}^d \lambda^iS_i \in \mathcal{G}_{2r}} \left\| \left[ \begin{array}{c}
P_0 \\
P_1  \\
 \vdots \\
P_{d}  \\
\end{array}\right] - \left[ \begin{array}{c}
S_0 \\
S_1  \\
 \vdots \\
S_{d}  \\
\end{array}\right] \right\|_F 
= \min_{\sum_{i= 0}^d \lambda^iS_i \in \mathcal{G}_{2r}} \left\| \left[ \begin{array}{c}
\vect(P_0) \\
\vect(P_1)  \\
 \vdots \\
\vect(P_{d})  \\
\end{array}\right] - \left[ \begin{array}{c}
\vect(S_0) \\
\vect(S_1)  \\
 \vdots \\
\vect(S_{d})  \\
\end{array}\right] \right\|_F,
\end{aligned}
\end{equation}
where $\vect (\cdot)$ denotes the operator that stacks the columns of a matrix into one long column vector \cite[Def. 4.2.9]{HoJo94}.
The reason for such reformulation is that the matrix polynomials $S(\lambda) \in \mathcal{G}_{2r}$ have a very special structure that we can use to develop an algorithm for solving the minimization problem. In the sequel, for simplicity, given a matrix polynomial $S(\lambda) = \sum_{i= 0}^d \lambda^iS_i $, we define
\begin{equation} \label{def.vecpoly}
 \vect (S(\lambda)) := \left[ \begin{array}{c}
\vect(S_0) \\
\vect(S_1)  \\
 \vdots \\
\vect(S_{d})  \\
\end{array}\right].
\end{equation}
From Theorem~\ref{thm.mainskew} every element in $\mathcal{G}_{2r}$ can be written as
$$S(\lambda)= U(\lambda)V(\lambda)^T - V(\lambda)U(\lambda)^T,$$
with $V(\lambda) \in \POL^{m\times r}_{\displaystyle \left\lceil d/2\right\rceil}$ and $U(\lambda) \in \POL^{m\times r}_{\displaystyle \left\lfloor d/2\right\rfloor}$, and, vice versa, every matrix polynomial of this form is in $\mathcal{G}_{2r}$.
Equivalently, 
\begin{equation} \label{linsyst}
\begin{aligned}
    \vect (S(\lambda))&= \vect (U(\lambda)V(\lambda)^T - V(\lambda)U(\lambda)^T) \\
    &= - {\bf M}(U(\lambda)) \, \vect(V(\lambda))  
    = {\bf M}(V(\lambda)) \, \vect(U(\lambda)), 
\end{aligned}
\end{equation}
where an $(d+1)m^2 \times (d-t+1)mr$ matrix ${\bf M}(W(\lambda))$ 
is defined for given matrix polynomials $W(\lambda) = \sum_{i= 0}^t \lambda^iW_i$ of size $m\times r$ and grade $t$, as follows 
\begin{equation}
    \begin{aligned}
&{\bf M} = {\bf M}(W(\lambda)) \\
&=
\scriptsize
\left[ \begin{array}{c c c c}
 (W_0 \otimes I) - (I \otimes W_0)N & 0 && 0 \\
 \vdots & (W_0 \otimes I) - (I \otimes W_0)N && \vdots  \\
 (W_{t-1}  \otimes I) - (I \otimes W_{t-1} )N   & \vdots && 0  \\
  (W_{t} \otimes I) - (I \otimes W_{t})N & (W_{t-1}  \otimes I) - (I \otimes W_{t-1} )N && (W_0 \otimes I) - (I \otimes W_0)N  \\
  0 & (W_{t} \otimes I) - (I \otimes W_{t})N  &\ddots& \vdots \\
  \vdots & 0 && (W_{t-1}  \otimes I) - (I \otimes W_{t-1} )N \\
  0 & \vdots && (W_{t} \otimes I) - (I \otimes W_{t})N  \\
\end{array}\right],
    \end{aligned}
\end{equation}
where N is the $mr \times mr$ permutation matrix that can ``transpose'' vectorized $m \times r$ matrices, i.e., $\vect(X^T)=N \vect(X) $ for any $m \times r$ matrix $X$ \cite[Th. 4.3.8]{HoJo94}, and the identity matrix has size $m\times m$.
Now we can rewrite \eqref{pro:minc} using \eqref{linsyst} in two different ways as follows:  
\begin{align} \label{pro:final1}
\dist_{\POL_{2r}}\left(\sum_{i= 0}^d \lambda^iP_i\right) &=  \min_{\scriptsize \begin{array}{l} 
V(\lambda) \in \POL^{m\times r}_{\left\lceil d/2\right\rceil}, \\[1mm] 
U(\lambda) \in \POL^{m\times r}_{\left\lfloor d/2\right\rfloor}
\end{array}} \left\| \vect(P(\lambda)) + {\bf M}(U(\lambda)) \, \vect(V(\lambda))  \right\|_F \\ 
&=  \min_{\scriptsize \begin{array}{l} 
V(\lambda) \in \POL^{m\times r}_{\left\lceil d/2\right\rceil}, \\[1mm]  
U(\lambda) \in \POL^{m\times r}_{\left\lfloor d/2\right\rfloor}
\end{array}} \left\| \vect(P(\lambda)) - {\bf M}(V(\lambda)) \, \vect(U(\lambda))\right\|_F. \label{pro:final2}
\end{align}
%
%
Following \cite{DDDG26}, we solve \eqref{pro:final1}-\eqref{pro:final2} using alternating least squares, see, e.g.,~\cite{LiLR16, SrJa03}. To be exact, we fix $U(\lambda)$ and use formulation \eqref{pro:final1} to find $V(\lambda)$ by solving
$$
 \min_{\vect(V(\lambda)) \in \mathbb{C}^{mr(\left\lceil d/2\right\rceil + 1)}} \left\| \vect(P(\lambda)) + {\bf M}(U(\lambda)) \, \vect(V(\lambda)) \right\|_F,
$$
then use the obtained $V(\lambda)$ in formulation \eqref{pro:final2} to find $U(\lambda)$ by solving 
$$
 \min_{\vect(U(\lambda)) \in \mathbb{C}^{mr(\left\lfloor d/2\right\rfloor + 1)}} \left\| \vect(P(\lambda)) - {\bf M}(V(\lambda)) \, \vect(U(\lambda))  \right\|,
$$
and repeat the two steps with the computed $U(\lambda)$ until convergence.
Moreover, solutions to these minimization problems are found simply by using Matlab's backslash command, i.e., $\vect(V(\lambda)) := -{\bf M}(U(\lambda)) \backslash \vect(P(\lambda))$ and $\vect(U(\lambda)) := {\bf M}(V(\lambda)) \backslash \vect(P(\lambda))$. 

Note that the sequence of distances $\rho_k := \left(\sum_{i=0}^d \|P_i - S_i^{(k)} \|_F^2 \right)^{1/2}$, where $S^{(k)} := U^{(k)} (\lambda) V^{(k)}  (\lambda)^T - V^{(k)}  (\lambda) U^{(k)}  (\lambda)^T$, generated by this alternating least squares approach is always convergent, since $\rho_k$ is a monotonic non-increasing sequence bounded below by zero.

\section{Skew-symmetric matrix pencils}
\label{pencils}  

This section is devoted to grade-one matrix polynomials, also known as matrix pencils. Here, we recall some definitions and results for skew-symmetric matrix pencils, in a similar way as was done for skew-symmetric matrix polynomials in Section \ref{sect.prempolys}. The reason for treating pencils separately is that our approach to solving the distance problems allows a significant improvement in this case; see Section \ref{sec:distpen} and Example \ref{ex:pen}.

We start by recalling the Kronecker-type canonical form of skew-symmetric matrix pencils under congruence. 
For each $k=1,2, \ldots $, { and for each $\mu \in \mathbb{C}$} define the $k\times k$
matrices
\begin{equation*}
J_k(\mu):=\begin{bmatrix}
\mu&1&&\\
&\mu&\ddots&\\
&&\ddots&1\\
&&&\mu
\end{bmatrix},\qquad
I_k:=\begin{bmatrix}
1&&&\\
&1&&\\
&&\ddots&\\
&&&1
\end{bmatrix},
\end{equation*}
and for each $k=0,1, \ldots $, define the $k\times
(k+1)$ matrices
\begin{equation*}
F_k :=
\begin{bmatrix}
0&1&&\\
&\ddots&\ddots&\\
&&0&1\\
\end{bmatrix}, \qquad
G_k :=
\begin{bmatrix}
1&0&&\\
&\ddots&\ddots&\\
&&1&0\\
\end{bmatrix}.
\end{equation*}
All non-specified entries of $J_k(\mu), I_k, F_k,$ and $G_k$ are zeros.

An $m \times m$ matrix pencil $A - \lambda B$ is called {\it congruent} to $C - \lambda D$ if there is a non-singular matrix $S$ such that $S^{T}AS =C$ and $S^{T}BS=D$.
We also define  the {\it congruence orbit} of a skew-symmetric matrix pencil $A - \lambda B$ under the action of the group $GL_m (\mathbb C)$ on the space of all skew-symmetric matrix pencils by congruence as follows:
\[
\orb^c (A - \lambda B) = \{S^{T} (A - \lambda B) S \ : \ S \in GL_m(\mathbb C)\}.
\]
%
In Theorem \ref{lkh} we recall the canonical form under congruence of skew-symmetric matrix pencils.
\begin{theorem}[\cite{Thom91}]\label{lkh}
Each skew-symmetric $m \times m$ matrix pencil $A - \lambda B$ is congruent
to a direct sum,
uniquely determined up
to permutation of
summands, of pencils of
the form
\begin{align*}
{\cal H}_{h}(\mu)&:=
\begin{bmatrix}0&J_h(\mu)\\
-J_h(\mu)^T &0
\end{bmatrix} - \lambda
\begin{bmatrix}0&I_h\\
-I_h &0
\end{bmatrix}
,\quad \mu \in\mathbb C,\\
{\cal K}_k&:=
\begin{bmatrix}0&I_k\\
-I_k&0
\end{bmatrix} - 
\lambda
\begin{bmatrix}0&J_k(0)\\
-J_k(0)^T &0
\end{bmatrix}, \\
{\cal M}_m&:=
\begin{bmatrix}0&F_m\\
-F_m^T &0
\end{bmatrix} - \lambda
\begin{bmatrix}0&G_m\\
-G_m^T&0
\end{bmatrix}.
\end{align*}
\end{theorem}
\noindent The blocks ${\cal H}_{h}(\mu)$ and ${\cal K}_k$ correspond to the finite and infinite eigenvalues, respectively, and all together form the {\em regular part} of $A - \lambda B$. The blocks ${\cal M}_m$ correspond to the right and left minimal indices, which are equal in the case of skew-symmetric matrix pencils, and form the {\em singular part} of $A - \lambda B$. Using this canonical form, we can now formulate Theorem \ref{thm:orbit}, which is a special case of Theorem \ref{anydth}.  
Define $\PEN_{2r,m}^{ss}$ as the set of $m \times m$ skew-symmetric matrix pencils of rank $\leq 2r$, i.e.,  $\PEN_{2r,m}^{ss}= \POL_{1, 2r, m}^{ss}$. 
\begin{theorem}[\cite{DmDo18}] \label{thm:orbit} 
    Let $m$ and $r$ be integers such that $2 \leq 2r \leq m-1$.
    Then  $\PEN_{2r,m}^{ss}=\overline{\operatorname{O}^c (\mathcal{W}^m_{2r}})$, where
    \begin{equation*}
        \mathcal{W}^m_{2r} = \operatorname{diag}(
        \underbrace{{\cal M}_{\alpha + 1}, \dots, {\cal M}_{\alpha + 1}}_{t},
        \underbrace{{\cal M}_\alpha, \dots, {\cal M}_\alpha}_{m-2r-t}),
    \end{equation*}
    with $\alpha = \lfloor \frac{r}{m-2r} \rfloor$ and $t = r\mod{(m-2r)}$. 
\end{theorem}
\noindent Similarly, the following theorem was proven earlier in \cite{DeMM22}. Note that, with the notation of Theorem \ref{thm.mainskew}, an immediate consequence of Theorem \ref{thm:decomp} is that $\PEN_{2r,m}^{ss} = \mathcal{G}^{ss}_{1,2r,m}$,  {\em without involving the closure of} $\mathcal{G}^{ss}_{1,2r,m}$. Thus, the result for skew-symmetric pencils is stronger than for skew-symmetric matrix polynomials of grade larger than $1$. It is possible to prove that the use of the closure is necessary in Theorem \ref{thm.mainskew}.

\begin{theorem}[Rank-1 decomposition of skew-symmetric matrix pencils from \cite{DeMM22}] \label{thm:decomp}
    Let $S_0 - \lambda S_1 \in \PEN_{2r,m}^{ss}$. Then
    \begin{equation*}
        S_0 - \lambda S_1 = \sum_{i=1}^r u_i v_i^T - v_i u_i^T,
    \end{equation*}
    for some vectors $u_1, \dots, u_r \in \mathbb{C}^m$ and polynomial vectors $v_1, \dots, v_r \in \mathbb{C} [\lambda]^m$ of degree $\leq 1$.
\end{theorem}
\noindent The decomposition in Theorem \ref{thm:decomp} is achieved by decomposing each block in the canonical form in Theorem \ref{lkh} into a sum of rank-1 matrix pencils (a tool which is not available for skew-symmetric matrix polynomials of grade larger than $1$). Also note that for a pencil $A - \lambda B$ \textit{not} in the canonical form, this decomposition is still possible. Let $S$ be a matrix such that $S^{-1} (A - \lambda B) S^{-T}$ is in the canonical form. Then, a rank-1 decomposition of $A - \lambda B$ is done by taking $v_i = S\tilde{v}_i$ and $u_i = S\tilde{u}_i$, where $\tilde{v}_i$ and $\tilde{u}_i$ are the vectors in the decomposition of the matrix pencil in the canonical form.

Note also that since skew-symmetric pencils always have even rank, the rank-1 pencils in Theorem \ref{thm:decomp} come in pairs ($u_i v_i^T$ and $v_i u_i^T$) that form skew-symmetric pencils ($u_i v_i^T - v_i u_i^T$) of rank 2, when the rank of $S_0 - \lambda S_1$ is exactly $2r$.

\section{Distance to rank-deficient skew-symmetric matrix pencils}
\label{sec:distpen}

%
The fact that the vectors $u_i$ in Theorem \ref{thm:decomp} are constant vectors allows us to improve considerably the algorithm described in Section \ref{sec:poldist} in the pencil case. Further improvements are achieved if the input pencil is skew-symmetric. Therefore, in this section we specifically deal with skew-symmetric input pencils. Given an $m \times m$ complex skew-symmetric matrix pencil $A - \lambda B$, we find a nearby skew-symmetric matrix pencil $S_0 - \lambda S_1$ of rank $\leq 2r$, with $2 \leq 2r < m$. 
By Theorem \ref{thm:decomp}, we can assume that $S_0 - \lambda S_1$ is in the form 
\[
    S_0 = U V_0^T - V_0 U^T \;\; \text{and} \;\; S_1 = U V_1^T - V_1 U^T,
\]
where $U=[u_1, \dots , u_r]$, and $V_0 - \lambda V_1 = [v_1, \dots , v_r]$.  Denoting for simplicity $\PEN^{ss}_{2r,m}$ just by $\PEN_{2r}$, our distance problem is reformulated in the pencil case as follows 
\begin{equation} \label{pro:mincpen}
\begin{aligned}
\dist_{\PEN_{2r}}\left(A - \lambda B\right) &= \min_{S_0 - \lambda S_1 \in \PEN_{2r}} \left(|| A -S_0 ||^2_F +|| B -S_1 ||^2_F \right)^{1/2} \\ 
&= \min_{U,V_0,V_1 \in \mathbb{C}^{m \times r}} \left\| \left[ \begin{array}{c}
A \\
B  \\
\end{array}\right] - \left[ \begin{array}{c}
U V_0^T -V_0 U^T \\
U V_1^T -V_1 U^T   \\
\end{array}\right] \right\|_F. 
\end{aligned}
\end{equation}
Similarly to Section \ref{sec:poldist}, we solve \eqref{pro:mincpen} using a method of alternating directions. As before, we fix $U$ to find $V_0$ and $V_1$, i.e., we solve
\begin{equation*}
\begin{aligned}
\min_{V_0,V_1 \in \mathbb{C}^{m \times r}} \left\| \left[ \begin{array}{c}
A \\
B \\
\end{array}\right] - \left[ \begin{array}{c}
U V_0^T -V_0 U^T \\
U V_1^T -V_1 U^T \\
\end{array}\right] \right\|_F, 
\end{aligned}
\end{equation*}
but now this problem is equivalent to solving the following two minimization problems independently 
\begin{equation} \label{split}
\min_{V_0\in \mathbb{C}^{m \times r}} \left\| A  - (U V_0^T -V_0 U^T) \right\|_F \quad \text{ and } \quad  \min_{V_1\in \mathbb{C}^{m \times r}} \left\| B  - (U V_1^T -V_1 U^T) \right\|_F. 
\end{equation}
Then, we use the obtained $V_0$ and $V_1$ to find $U$, i.e., we solve 
\begin{equation} \label{pen.u}
\begin{aligned}
\min_{U \in \mathbb{C}^{m \times r}} \left\| \left[ \begin{array}{c}
A \\
B \\
\end{array}\right] - \left[ \begin{array}{c}
U V_0^T -V_0 U^T \\
U V_1^T -V_1 U^T \\
\end{array}\right] \right\|_F.  
\end{aligned}
\end{equation}
Now we explain how to solve each of the problems in \eqref{split}. 

We assume for this purpose that $U$ has full rank $r$, which is the generic situation in the algorithm, but we will make some comments at the end on how to proceed when $U$ is rank deficient. Note that, taking the SVD of $U= RST^H$,  with $R \in \mathbb{C}^{m\times m}$ and $T \in \mathbb{C}^{r\times r}$ unitary (orthogonal in the real case), yields the following minimization problem equivalent to the first one in \eqref{split}
\begin{equation}\label{eq:svd} 
\min_{V_0\in \mathbb{C}^{m \times r}} \left\|  R^H A \overline{R} - (S (R^H V_0 \overline{T})^T - (R^H V_0 \overline{T}) S^T) \right\|_F . 
\end{equation}
An analogous problem can be obtained for the second minimization problem  for $V_1$ in \eqref{split} with $R^H B \overline{R}$ instead of $R^H A \overline{R}$.
Let $X := R^H V_0 \overline{T}$ and $E := R^H A \overline{R}$, and note that $E$ is skew-symmetric. Denote by $S_1$ the $r \times r$ diagonal block of $S$ and by $X_1$ and $X_2$ the $r$ first and $m-r$ last rows of $X$, that is,
\[
    S = \begin{bmatrix}
        S_1 \\
        0 \\
    \end{bmatrix} \quad \text{ and } \quad 
        X = \begin{bmatrix}
        X_1 \\
        X_2 \\
    \end{bmatrix}.
\]
Note that $S_1$ is invertible if $U$ has rank $r$. By making the corresponding partition of $E =  \left[ \begin{smallmatrix}
        E_{11} & E_{12} \\
        E_{21} & E_{22} \\
    \end{smallmatrix} \right]$, we obtain from \eqref{eq:svd} the following equivalent minimization problem
\begin{equation} \label{eq:svd_mat} 
\min_{X_1 \in \mathbb{C}^{r\times r}, X_2 \in \mathbb{C}^{(m-r)\times r}}
\left\|\begin{bmatrix}
        E_{11} & E_{12} \\
        E_{21} & E_{22} \\
    \end{bmatrix} - \begin{bmatrix}
        S_1 X_1^T - X_1 S_1^T & S_1 X_2^T \\
        -X_2 S_1^T & 0 \\
    \end{bmatrix}\right\|_F.
\end{equation}
The problems for $X_1$ and $X_2$ decouple in \eqref{eq:svd_mat} and can be solved very efficiently by computing the exact solutions of two linear systems as we explain in the following. Since $E_{21} = -E_{12}^T$, the two off-diagonal blocks that contain $X_2$ are equivalent and $X_2 =  E_{12}^T S_1^{-1}$, which taking into account that $S_1$ is diagonal can be computed very efficiently.
As for $X_1$, we note that since $S_1$ is diagonal and $E_{11}$ is skew-symmetric, the $(i,j)$-element in $S_1 X_1^T - X_1 S_1^T = E_{11}$ is $s_{ii} x_{ji} - s_{jj} x_{ij} = e_{ij}$, while the $(j,i)$-element is  $s_{jj} x_{ij} - s_{ii} x_{ji} = e_{ji} = -e_{ij}$. So, both lead to the same equation for $x_{ij}$ and $x_{ji}$, when $i\ne j$, which can be chosen in infinitely many ways (as well as the diagonal elements of $X_1$ since there are no restriction over them).  In particular, we can choose $X_1$ to be skew-symmetric. Then, the equations simplify to $x_{ji} = \frac{e_{ij}}{s_{ii} + s_{jj}}$. When $X$ is calculated, $V_0$ is given by
\[
    V_0 = R \begin{bmatrix} X_1 \\ X_2 \end{bmatrix} T^T.
\]
In the same fashion, we obtain $V_1$. If the rank of $U$ is less than $r$, then some diagonal entries of $S_1$ are zero, but \eqref{eq:svd_mat} is still valid. In that scenario, we can use the Moore-Penrose inverse, $S_1^\dagger$, of $S_1$, which is a diagonal $r\times r$ matrix with $(S_1^\dagger)_{ii} = 1/s_{ii}$, when $s_{ii} \ne  0$, and zeros elsewhere. Then, $X_2 =  E_{12}^T S_1^{\dagger}$ is a least squares solution with minimum norm for $X_2$. A least squares minimal norm skew-symmetric solution for $X_1$ is obtained taking $x_{ji} =0$, when $s_{ii} = s_{jj} = 0$, and proceeding as in the full rank case for the other entries.

Now, fixing $V_0$ and $V_1$, we take another step of the alternating directions. We solve for $U$ in exactly the same way as in the case of polynomials of any degree derived in Section \ref{sec:poldist}, i.e.,   
\[
 \vect(U) :=   \begin{bmatrix} 
 \left( V_0 \otimes I_m \right) - \left( I_m\otimes V_0 \right)N  \\
  \left( V_1 \otimes I_m \right) - \left( I_m \otimes V_1 \right)N  \\
 \end{bmatrix} \ \backslash \ 
 \begin{bmatrix}
        \vect (A) \\
        \vect (B) \\
 \end{bmatrix}.
\]

\section{Numerical experiments} \label{sect.num}

In this section, we compare the developed algorithms with one another and, when possible, with other available algorithms. Specifically, we compute the distance to singularity for matrix polynomials using three methods (including ours), and we compute the distance to the set of matrix pencils of a specified bounded rank using two of our algorithms. The reason for this distinction is that the other available algorithms can compute only the distance to singularity, and not the distance to the nearest matrix polynomial of rank $2r$ (for arbitrary $r$). The experiments were run using MATLAB (R2024b Update 4) on MacBook Air (Apple M3, 16GB). The code is available on GitHub.\footnote{\url{https://github.com/rakeljh/dist_sing_skew_pol}}

Our algorithm uses Generic Eigenstructures and Alternating direction method for Rank $2r$ approximation of Skew-symmetric matrix polynomials, thus we abbreviate it to GEARS. We also use GEARS-SVD for the version with svd-decomposition for matrix pencils.   

\begin{example} \label{ex:pol}

We calculate a distance to singularity for skew-symmetric matrix polynomials using three different algorithms. Namely Riemman-Oracle (RO) \cite{GNNP24}  and Matrix-Valued Function (MVF) \cite{GnGu23,GnGu25}, as well as GEARS, the algorithm developed in this paper. We run each algorithm on 72 random skew-symmetric matrix polynomials, i.e., 4 random matrix polynomials for each combination of the size ($4 \times 4$, $6 \times 6$, or $8 \times 8$), degree (1, 2, or 3), and the underlying field (real or complex). The results are summarized in Figures~\ref{fig:4-6-8-1},\ref{fig:4-6-8-2}, and \ref{fig:4-6-8-3}. The graphs in Figure \ref{fig:4-6-8-1} show that all algorithms find a skew-symmetric matrix polynomial at essentially the same distance from a given skew-symmetric matrix polynomial. Nevertheless, our algorithm produces the answers that are ``the most singular''\footnote{We note that the MVF algorithm requires an assumption \cite[Assumption~4.8]{GnGu25}. This assumption used in the outer iteration of the method is that the smallest singular value of $(F + \Delta F)(\mu_j)$ is simple and non-zero for $j = 1, \dots , m$. In the case of skew-symmetric matrix polynomials, the singular values come in pairs and hence the smallest singular value is only simple if $n$ is odd, in which case it is zero. This may affect the performance of this method when used to find the nearest singular skew-symmetric matrix pencil.}, see Figure \ref{fig:4-6-8-2} (recall that the output of GEARS is singular by construction). Moreover, Figure \ref{fig:4-6-8-3} shows that GEARS is also the fastest out of the tested algorithms.  

\begin{figure}[ht]
    \hspace{-1.7cm}     \includegraphics[width=1.1\textwidth]{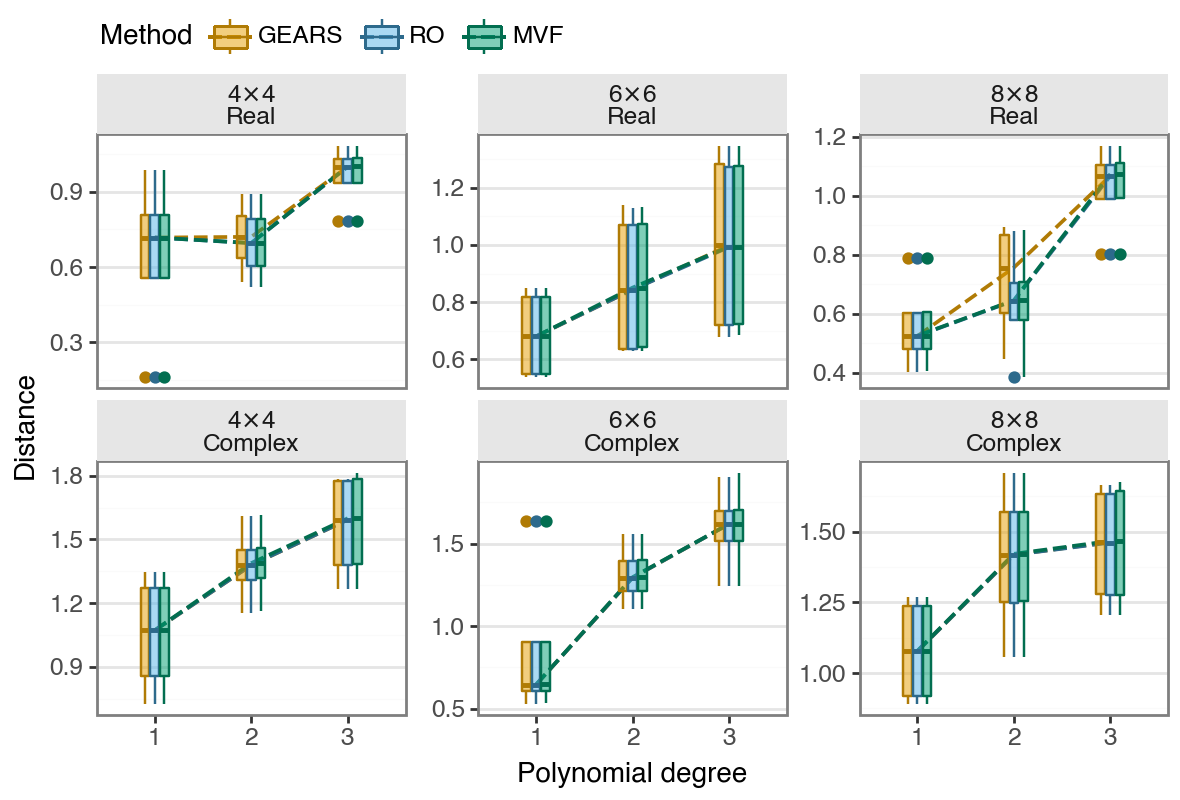}
    \caption{Comparison of the computed minimal distances to singularity for random skew-symmetric matrix polynomials of given size, degree, and with the entries in the underlying field (real or complex). All the algorithms find a skew-symmetric matrix polynomial at essentially the same distance from a given skew-symmetric matrix polynomial.
    }
    \label{fig:4-6-8-1}
\end{figure}

\begin{figure}[ht]
    \hspace{-1.7cm}         \includegraphics[width=1.1\textwidth]{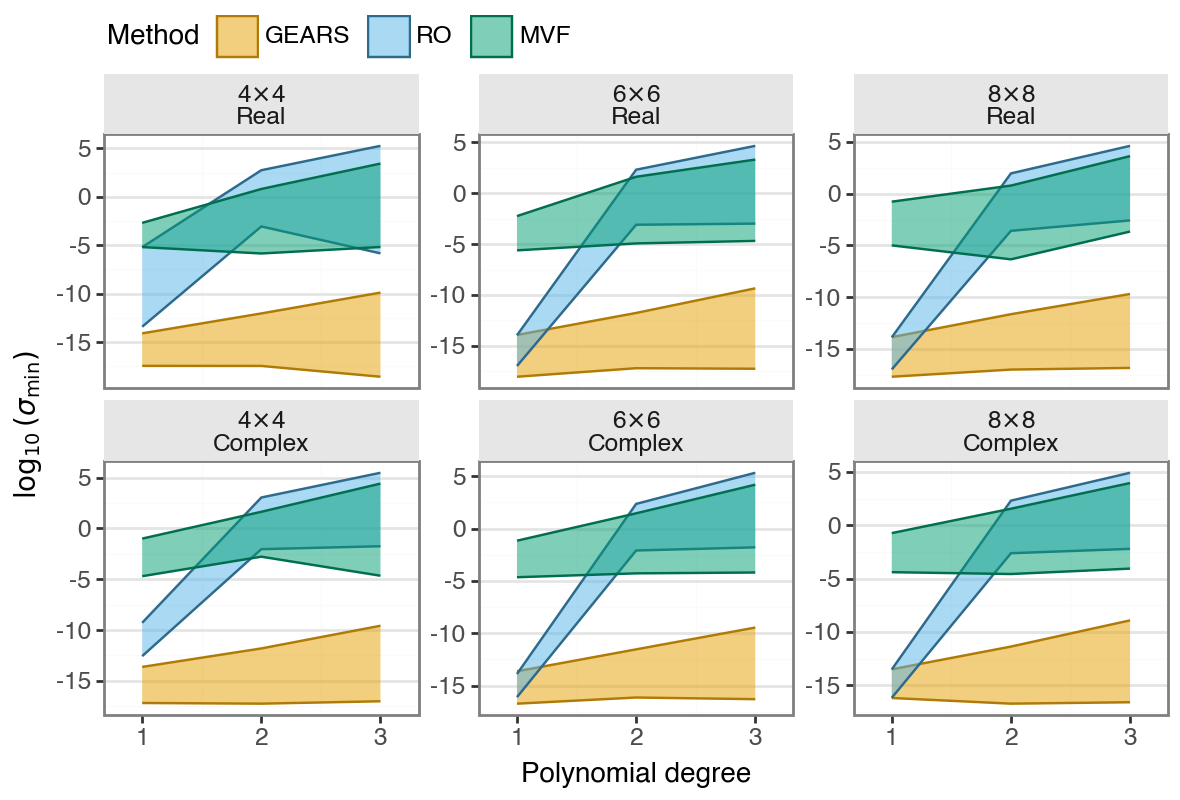}
    \caption{Comparison of the ``singularity'' of the output for random skew-symmetric matrix polynomials of given size, degree, and the entries in the underlying field (real or complex). To be exact, we follow \cite{GnGu25} and plot smallest singular value $\sigma_{\min}$ of the constructed matrix polynomial evaluated at the points ${x_j + \sqrt{-1} y_j}$ given by 
    $[X,Y] = meshgrid(-1000:40:1000)$. In all the performed experiments GEARS produces the outputs with significantly lower singular values. 
    }
    \label{fig:4-6-8-2}
\end{figure}

\begin{figure}[ht]
    \hspace{-1.7cm}     \includegraphics[width=1.1\textwidth]{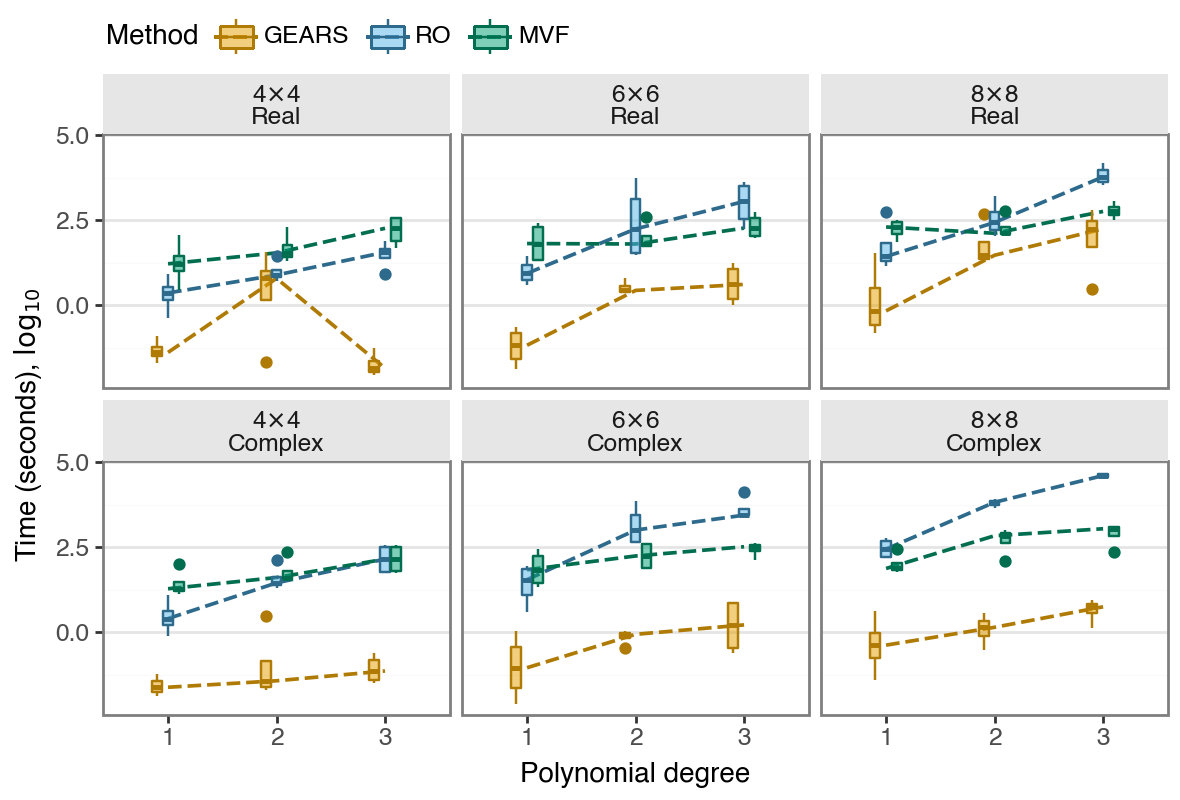}
    \caption{Comparison of the speed for random skew-symmetric matrix polynomials of given size, degree, and the entries in the underlying field (real or complex). In all the performed experiments GEARS is the fastest. 
    }
    \label{fig:4-6-8-3}
\end{figure}

In addition, to investigate how the computational cost of GEARS scales with the input size, we extend our experiments to random skew-symmetric matrix polynomials of degree $2$ with sizes ranging from $20\times20$ up to $50\times50$, over both the real and complex fields, Figures \ref{fig:deg2-large-dist} and \ref{fig:deg2-large-time}. 
We carefully controlled the quality of the computed output, roughly matching the accuracy reported in \cite{GNNP24} for $30\times30$ complex unstructured matrix polynomials of degree $2$.\footnote{To be exact we terminate the algorithm when the change in distance between consecutive iterations is less than $10^{-4}\cdot \|P\|_F$ for complex matrix polynomials and less than $10^{-3}\cdot \|P\|_F$ for real matrix polynomials.} Note that \cite{GNNP24} reports the runtime of around 300 seconds for $30\times30$ matrices, while our approach requires approximately 200 seconds for comparable output quality. In the real case, the algorithm is typically about one order of magnitude faster than in the complex setting when asked to maintain the same quality of the output. 
\begin{figure}[ht]
\includegraphics[width=0.98\textwidth]{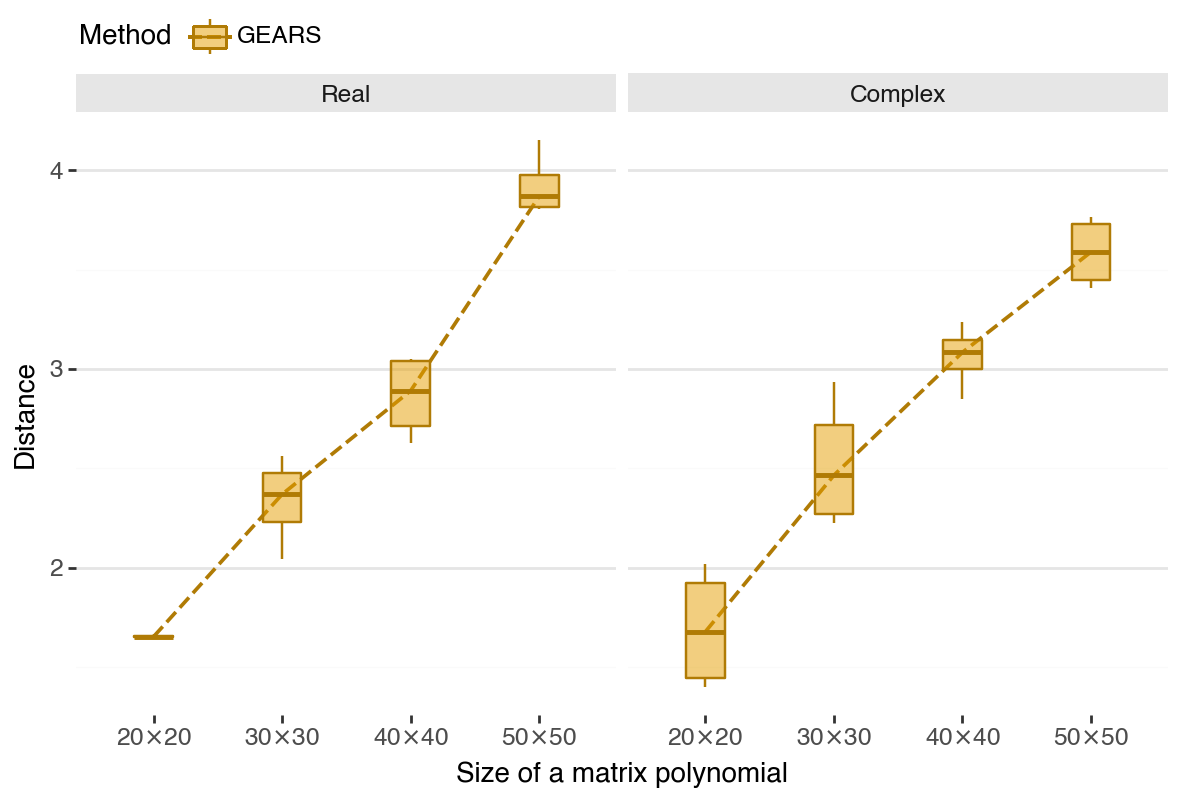}
    \caption{GEARS algorithm for computing the minimal distance to singularity for random skew-symmetric matrix polynomials of sizes $20\times20$, $30\times30$, $40\times40$, and $50\times50$, of degree $2$, with entries over the real or complex field (4 examples for each setting).}
    \label{fig:deg2-large-dist}
\end{figure}

\begin{figure}[ht]
    \includegraphics[width=0.98\textwidth]{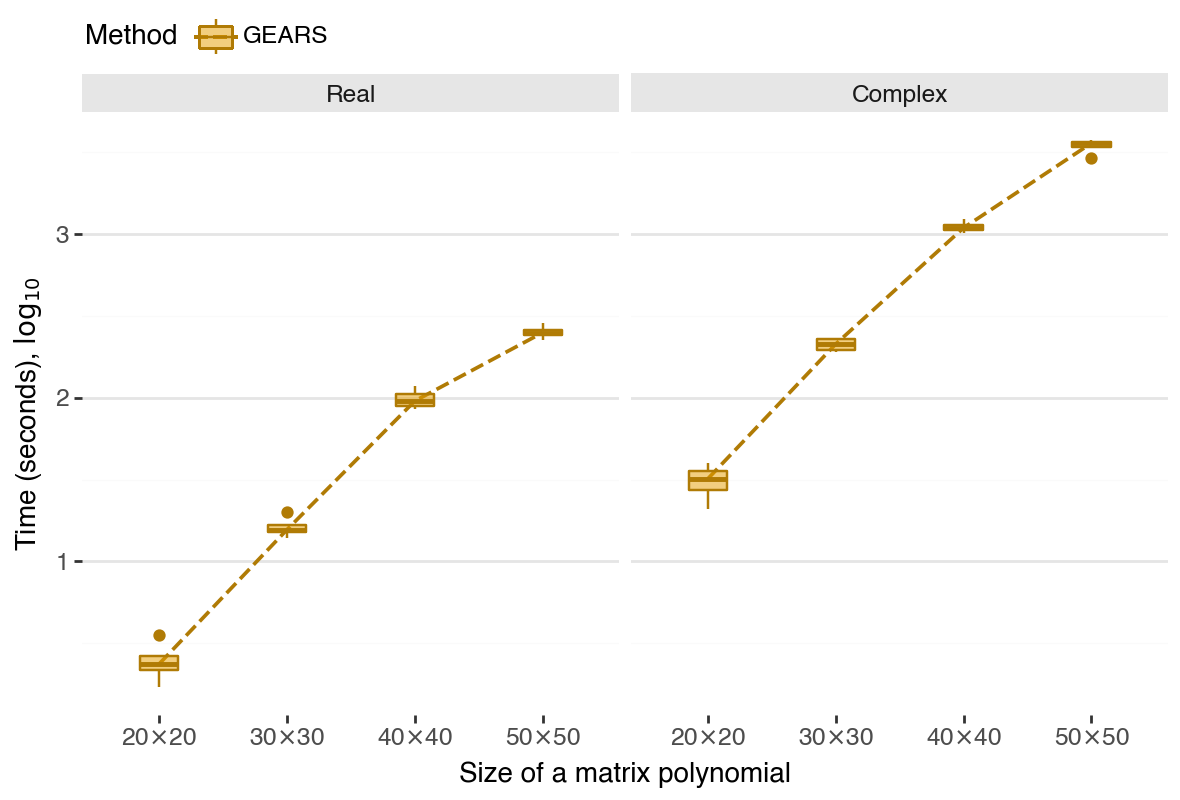} 
        \caption{Illustration of the time scaling of the GEARS algorithm as the input size increases. The plot shows the runtime for random skew-symmetric matrix polynomials of given sizes and degree $2$, with entries over the real or complex field (4 examples for each setting).}
    \label{fig:deg2-large-time}
\end{figure}

\end{example}

In Example \ref{ex:pol}, we use the code for polynomials, which works for any degree, including degree one. Nevertheless, Section \ref{pencils} and Theorem \ref{thm:decomp} show that the parametrization for matrix polynomials in Theorem \ref{thm.mainskew} simplifies significantly when we are dealing with matrix pencils. This has enabled us to develop the approach explained in Section \ref{sec:distpen}, which for one of the alternating directions step solves matrix equations directly, without reformulating it using the Kronecker product as a large system of linear equations. Figures \ref{fig:pencil_dist} and \ref{fig:pencil_time} illustrate the resulting time savings while maintaining accuracy for the skew-symmetric pencils generated in Example \ref{ex:pen}.

\begin{example} \label{ex:pen}
In this example we compare performance of the algorithms GEARS and GEARS-SVD developed in Sections \ref{sec:poldist} and \ref{sec:distpen}, respectively, on the matrix pencil input. We run each algorithm on 50 random skew-symmetric matrix pencils, i.e., 10 random complex skew-symmetric matrix pencils for each size (from $6 \times 6$ to $30 \times 30$), and compute the result for ranks $n-2$ for all sizes, and $\lfloor n/4 \rfloor$ for sizes $12 \times 12$ and up. The results are summarized in Figures \ref{fig:pencil_dist} and \ref{fig:pencil_time}. 

 As seen in Figure \ref{fig:pencil_time}, the computation time was improved when using GEARS-SVD compared to GEARS. The experiments were done with a maximum number of iterations, and since GEARS-SVD converges faster than GEARS, the distance results were also slightly better for large matrix pencil sizes when using GEARS-SVD. This is seen in Figure \ref{fig:pencil_dist}.

\begin{figure}
    \centering
    \includegraphics[width=0.92\linewidth]{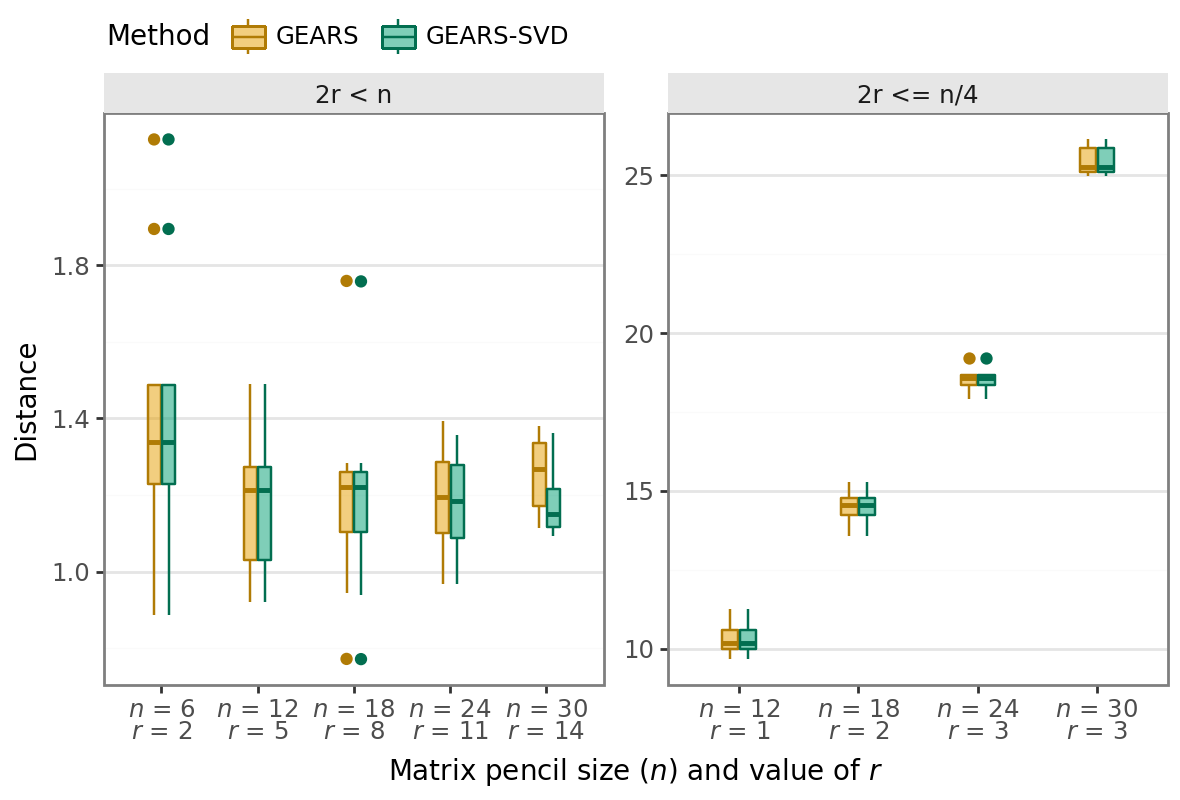}
    \caption{Comparison of the computed minimal distances for random complex skew-symmetric matrix pencils of size $n \times n$ to the set of skew-symmetric pencils with rank $\leq 2r$. The only noticeable difference is for the largest sizes.} 
    \label{fig:pencil_dist}
\end{figure}

\begin{figure}
    \centering
    \includegraphics[width=0.92\linewidth]{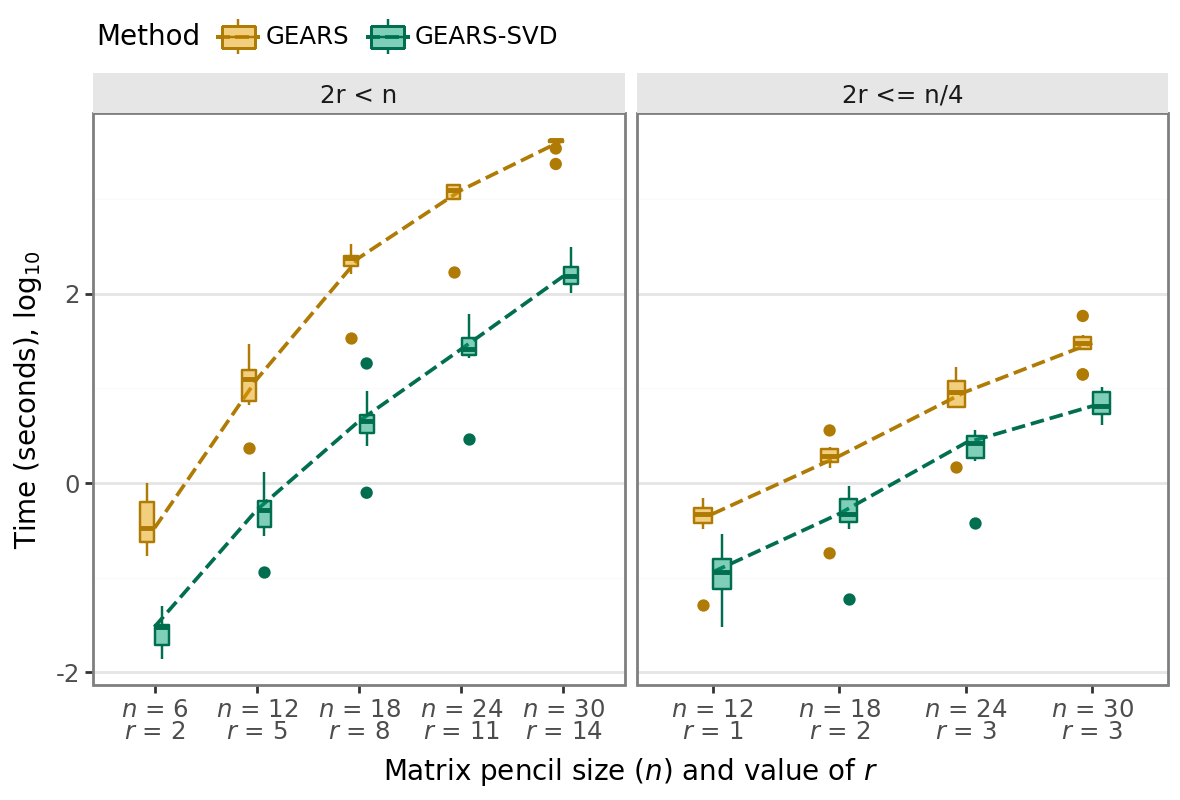}
    \caption{
    Comparison of the computation time for determining the distance to the set of skew-symmetric matrix pencils with rank $\leq 2r$, for random complex skew-symmetric matrix pencils of size $n \times n$.
    GEARS-SVD outperforms GEARS.}
    \label{fig:pencil_time}
\end{figure}
 
\end{example}

\section{Conclusions and future work} \label{sect.conclusions}

We present an algorithm that approximates a given matrix polynomial by a skew-symmetric matrix polynomial of a specified rank. Our algorithm works for any rank $2r$ such that $2 \leq 2r \leq m-1$, where $m \times m$ is the matrix size, and produces a skew-symmetric matrix polynomial that is of rank $2r$ up to rounding errors. In the tested settings, it is significantly faster than the other available algorithms, which in addition are only valid for the particular case when $2r$ in the largest even number smaller than $m$.

Extensions of this result to other classes of structured matrix polynomials, such as symmetric, symmetric/skew-symmetric, or palindromic, are part of our future work.

\medskip

\noindent{\bf Acknowledgements.} 

The work of A. Dmytryshyn was supported by the Swedish Research Council (VR) grant 2021-05393.
The work of F. M. Dopico was partially supported by the grant PID2023-147366NB-I00 funded by the Agencia Estatal de Investigaci\'on of Spain MICIU/AEI/10.13039/501100011033 and FEDER/UE.

{\footnotesize
\bibliographystyle{abbrv}
\bibliography{library}
}

\end{document}